\documentclass[12pt]{amsart}

\usepackage{amsfonts,amsmath,amssymb,amsthm,latexsym,verbatim,amscd}
\usepackage{graphics}
\usepackage{bbm}
\usepackage{changes}

\newcommand{\bbT}{\mathbb{T}}

\newcommand{\calC}{\mathcal{C}}

\newcommand{\calH}{\mathcal{H}}

\newcommand{\calM}{\mathcal{M}}

\theoremstyle{plain}
\numberwithin{equation}{section}

\def\a{{\alpha}}

\def\S{{\mathcal S}}

\newcommand{\R}{{\mathbb R}}
\newcommand{\Sph}{{\mathbb S{\rm ph} }}
\newcommand{\SY}{{\mathbb S }}
\newcommand{\HY}{{\mathbb H}}
\newcommand{\TY}{{\mathbb T{\rm or} }}
\newcommand{\N}{{\mathbb N}}

\newcommand{\Z}{{\mathbb Z}}

\theoremstyle{plain}
\newtheorem{theorem}{Theorem}[section]
\newtheorem*{theorem*}{Theorem}

\newtheorem*{corollary*}{Corollary}
\newtheorem{proposition}[theorem]{Proposition}
\newtheorem*{proposition*}{Proposition}
\newtheorem{lemmad}[theorem]{Dragging Lemma }
\newtheorem{lemma}[theorem]{Lemma}
\newtheorem*{lemma*}{Lemma}

\newtheorem*{example*}{Example}

\newtheorem*{definition*}{Definition}

\newtheorem*{notation*}{Notation}
\newtheorem{remark}[theorem]{Remark}
\newtheorem*{remark*}{Remark}

\title[Minimal Surfaces in Finite Volume Hyperbolic 3-Manifolds]{Minimal Surfaces in Finite Volume Hyperbolic 3-Manifolds $N$ and in $M \times \SY^1$, $M$ a Finite Area Hyperbolic Surface.}

\author{P. Collin}
\address{P.Collin, Institut de math\'ematiques de Toulouse,
Universit\'e Paul Sabatier, 118, route de Narbonne, 31062 Toulouse cedex ,France}
\email{collin@math.ups-tlse.fr   }

\author{L. Hauswirth}
\address{L. Hauswirth, Universit\'e Paris-Est, LAMA (UMR 8050), UPEMLV, UPEC, CNRS, F-77454, Marne-la-Vall\'ee, France}
\email{hauswirth@univ-mlv.fr}

\author{H. Rosenberg}
\address{H. Rosenberg, Instituto Nacional de Matematica Pura e Aplicada (IMPA) Estrada Dona Castorina 110, 22460-320, Rio de Janeiro-RJ, Brazil}
\email{rosen@impa.br}

\begin{document}

\begin{abstract} We consider properly immersed finite topology minimal surfaces $\Sigma$ in complete finite volume hyperbolic 3-manifolds $N$, and in $M  \times \SY^1$, where $M$ is a complete hyperbolic surface of finite area.  We prove $\Sigma$ has finite total curvature equal to $2 \pi$ times the Euler characteristic $\chi (\Sigma)$ of $\Sigma$, and we describe the geometry of the ends of $\Sigma$.
.
\end{abstract}

\thanks{{\it The authors were partially supported by the ANR-11-IS01-0002 grant.} \today}
\maketitle

\section{Introduction}
Let N denote a complete hyperbolic 3-manifold of finite volume.  An end ${\calM}$ of $N$ is modeled on the quotient of a horoball of the hyperbolic 3-space $\HY^3$, by a $\Z^2$ parabolic subgroup of the isometry group of $\HY^3$ leaving the horoball  invariant.  More precisely we consider the model of the half-space of $\HY^3=\{ (x,y,t) \in \R^3; y>0 \}$ with the metric $ds^2=\frac{dx^2 +dy^2+ dt^2}{y^2}$. Then an end of $N$ has a sub-end isometric to
$${\calM}(-1)= \{ (x,y,t) \in \R^3; y \geq y_0 >0\}$$ 
modulo a $\Z^2$-parabolic subgroup of isometries of $\HY^3$ leaving the planes $\{y=c\}$ invariant. The horosphere
$y$=constant quotient to tori $\bbT (y)$ in ${ \calM}(-1)$; $\bbT (y)$ has constant mean curvature one. Let
$c$ be a compact geodesic of $\bbT (1)$. Then $A(-1)=\{ (c,t); t \geq 1 \}$ is a minimal annulus immersed in
${ \calM}(-1)$, which we will call a standard cusp-end in ${ \calM}(-1)$

A complete surface  $M$ of constant curvature $K=-1$ and finite area has finite total curvature hence $M$ is conformally
diffeomorphic to a compact surface punctured in a finite number of points. Each end of $M$ (called a cusp end), denoted
$\calC$, is an annular end isometric to the quotient of a horodisk $H$ in the hyperbolic plane $\HY^2$
by a parabolic isometry $\psi$.

To describe the geometry of such ends we model $\HY^{2}$ by the upper half plane 
$$\HY^2 =\{ (x,y) \in \R^{2} ; y > 0 \}$$
with the metric $ds^2= \frac{dx^2+dy^2}{y^2}$. Then a cusp end $\calC$ of $M$ is isometric to $H / [\psi]$,
where $H= \{ (x,y) \in \R^2 ; y\geq 1\}$ is a horodisk and $\psi (x,y)=(x+\tau,y)$, for some $\tau \neq 0$.  

In $M \times \SY^1$, with the product metric, the ends become $\calM:=\calC \times \SY^1$, and are foliated by constant mean curvature tori $\bbT (y_1)=c(y_1) \times \SY^1$, where $c(y_1)= \{ (x,y) \in \HY^2; y=y_1 \} / [\psi]$. We consider $\SY^1 = \R / T(h)$, $T(h)$ the translation of $\R$ by some $h >0$ and
$$\calM=\cup_{y \geq y_0} \bbT (y)=( H / [\psi] ) \times (\R / T(h))=\{ (x,y,t) \in \R^3 ; y \geq y_0 \geq 1 \} / [\psi, T(h)].$$

Thus the ends of $N$ and those of $M \times \SY^1$ share many properties.  Both are parametrized by the same half-space of $\R^3$, and foliated by constant mean curvature tori $\bbT(y)$ ( curvature one half in $\calM$ and one in 
$\calM(-1)$). $\calM(-1)$ has constant sectional curvature $-1$ and the tori $\bbT(y)$ shrink exponentially when one flows by the geodesics y increasing.
In $\calM$, the horizontal cycles $c(y)$ shrink exponentially along the $y$ increasing flow and the $t$ cycles are of constant length $h$.  Subsequently we will develop the geometry of surfaces in these ends.

Now let $\Sigma$ be a properly embedded minimal surface in $N$ or $M \times \SY^1$ of finite topology; so that $\Sigma$ has a finite number of annular ends $\{ A_j\}$ for $1\leq j \leq k$.  Since $\Sigma$ is proper, each end $A_j$ of $\Sigma$ is in some end $\calM$ of $M \times \SY^1$ or in some end $\calM (-1)$ of $N$. We denote by $E$ a connected component of a lift of an end $A$   of $\Sigma$, $E$  in $\HY \times \R$ or $\HY^3$. 

We will now describe the model ends of minimal annuli in $\calM$ and $\calM (-1)$. In $\calM (-1)$ the model
end is the standard cusp end $A(-1)$ we previously defined.

 In $\calM$, there are essentially three model ends. In $\calM$, we define $A_{(p,q)}$ to be the annular end that is the quotient of a (euclidean)
half-plane $E_{(p,q)}$ orthogonal to the plane $\{ (x,y,t) \in \R^3; y=1\}$ and of slope $qh/p\tau$.
For $(p,q)=(1,0)$, the end 
$$E_{(1,0)} (t_0)=  \{ (x,y,t) \in \R^3 ; y \geq 1, t=t_0 \} \hbox{ and } A_{(1,0)}=E_{(p,q)} /[\psi]$$
 is a cusp end of $M$(horizontal). For $(p,q)=(0,1)$, it is the product of a horizontal geodesic ray of $M$ and $\SY^1$. The end 
 $$E_{(0,1)} (x_0)= \{ (x,y,t) \in \R^3 ; y \geq 1, x=x_0 \} \hbox{ and } A_{(0,1)}=E_{(0,1)}/T(h).$$
 For $(p,q) \neq \{ (0,1), (1,0)\}$, we think of $A_{(p,q)}$ as a helicoid with axis at the cusp at infinity. It is the quotient  of
$$E_{(p,q)} (c_0)= \{ (x,y,t) \in \R^3 ; y \geq 1 , p\tau t - qh x =c_0 \} \hbox{ and } A_{(p,q)}=E_{(p,q)}/[\psi, T(h)].$$

We will prove that a properly immersed annular end $A$ in $\calM$ or in $\calM (-1)$ has finite total curvature and is asymptotic to a standard end $A_{(p,q)}$ in $\calM$ or a standard cusp end $A(-1)$ in $\calM (-1)$.
The main theorem of the paper is:

\begin{theorem}
\label{Maintheorem}
Consider a complete surface $M$ with curvature $K=-1$ and finite area and $N$ a complete hyperbolic $3$-manifold of finite volume. Let $\Sigma$ be a properly immersed minimal surfaces in $N$ or in $M \times \SY^{1}$ with finite topology.
Then the surface $\Sigma$ has finite total curvature and each end $A$ of $\Sigma$ is asymptotic to a standard cusp-end
$A(-1)$ in $\calM (-1)$ or to a standard end $A_{(p,q)}$ in $\calM$:

\begin{enumerate}
\item[(i)] $A_{(p,0)}$ a horizontal cusp $\calC \times \{t_{0}\}$
\item[(ii)] $A_{(0,q)}$ a vertical plane $\gamma \times \SY^{1}$
\item[(iii)] $A_{(p,q)}$ a helicoidal end with axis at infinity.
\end{enumerate}
Moreover 
$$\int_{\Sigma} K dA = 2 \pi \chi (\Sigma).$$ 
\end{theorem}

{\bf Corollaries of the theorem \ref{Maintheorem}} Combining the formula for the total curvature
of $\Sigma$ in theorem \ref{Maintheorem}, with the Gauss equation, we obtain topological
obstructions for the existence of proper minimal immersions of a finite topology
surface $\Sigma$ into $N$ or $M \times \SY^1$, of a given topology. 

For example, there is no such proper minimal immersion of a plane $\R^2$ into $N$
or $M \times \SY^1$. In $N$ there is no proper minimal immersion of the sphere $\SY^2$
with $n$ punctures; $n=0,1,$ or $2$. A proper minimal immersion of $\SY^2$ with two punctures
(an annulus) in $M \times \SY^1$, is necessarily $\gamma \times \SY^1$, $\gamma$ a complete
geodesic of $M$.

More generally, suppose $\Sigma$ is an orientable surface of genus $g$ with $n$ punctures, $n \geq 0$.
Then $\chi (\Sigma)=2-2g-n$, so if $\Sigma$ can be properly
minimally immersed in $N$ or $M \times \SY^1$, it follows from theorem \ref{Maintheorem}
and the Gauss equation

$$\int_\Sigma K_\Sigma= 2 \pi (2-2g -n)= \int_\Sigma K_e + \int_\Sigma K_\sigma,$$
where $K_e$ and $K_\sigma$ are the extrinsic and sectional curvatures of $\Sigma$ respectively.
Since $-1 \leq K_\sigma \leq 0$ in $M \times \SY^1$, $K_\sigma =-1$ in $N$, and $K_e \leq 0$,
we have
$$2-2g-n \leq 0$$
and equality if and only if $K_e= K_\sigma =0$.

This equality cannot occur in $N$ (since $K_\sigma =-1$) and equality in $M \times \SY^1$
yields $\Sigma$ is vertical and $g$ is $0$ or $1$. When $g=0$, then $n=2$ and $\Sigma = \gamma \times \SY^1$ $\gamma$ a complete, non compact geodesic of $M$.

When $2<2g +n$ then if $g=0$, one can not have $n \leq 2$. So excluding the equality case we
discussed above, there is no proper minimal immersion of $\SY^2$ with $0,1,$ or $2$ punctures, in $N$
or $M \times \SY^1$.

In $N$, one obtains an area estimate. If $\Sigma$ is properly minimally immersed in $N$ then
$$2\pi (2-2g -n) = \int_\Sigma K_e - |\Sigma|,$$
so $|\Sigma|= \int_\Sigma K_e + 2 \pi (2g + n -2) \leq 2 \pi (2g + n -2)$ and equality precisely when
$\Sigma$ is totally geodesic. Do such totally geodesic immersions exists in $N$?

We have $0< 2g+n-2$, so if $2g+n-2 \leq 0$, the immersion $\Sigma$ does not exist in $N$. If
$2g > 2-n$, can $\Sigma$ be properly minimally immersed in $N$?.

The paper is organized as follows.  We will begin considering surfaces in $M \times \SY^1$.   First we describe some examples of properly embedded minimal surfaces
of finite topology in $M \times \SY^1$. We start with $M$ a 3-punctured sphere, then $M$ a sphere with $2n$
punctures, and $M$ a once punctured torus. We hope to convey to the reader the wealth of interesting examples
in these spaces.

In section 2, we describe some properties of the standard examples $A_{(p,q)}$ in the cusp ends $\calM$ of $M \times \SY^1$.
We construct auxiliary minimal surfaces needed for the sequel.

In section 3 we begin the study  of a lift $E \subset H \times\R$ of an annular end $A$ of $\Sigma \cap \calM$. We prove that a subend of $A$ is trapped between two standard ends $A_{(p,q)}$ that are close at infinity; "close" will be defined later.

In section 4, we study compact annuli that we will use in the proof of the theorem. In section 5, we study the limit of a family of
Scherk type graphs in $\HY^2$ which are converging to $0$ and we use this sequence to prove that
the third coordinate of an end of type $(1,0)$ has a limit at infinity.

In section 6, we prove that a trapped subend of $A$ is a killing graph, hence has bounded curvature. 
Then in sections 7,8, and 9, we prove the main theorem.

\section{Examples in $M \times \SY^1$}

The first examples in $M \times \SY^1$ that come to mind are the horizontal slices
$\Sigma=M \times \{ c\}$ and the vertical annuli (or totally geodesic tori), $\Sigma= \gamma \times \SY^1$,
$\gamma$ a complete geodesic (perhaps compact) of $M$.

We describe five examples; $M$ will be a sphere with three or four punctures or a once punctured torus,
and have a complete hyperbolic metric of finite area. Denote by $\Sph (k), k=3$ or $4$ such a hyperbolic
sphere and $\TY (1)$ a once punctured hyperbolic torus.

\vskip 0.2cm
{\bf Example 1.} $\Sigma$ an embedded minimal surface in $\Sph (3) \times \SY^1$ with three ends;
two helicoidal and the other horizontal. The domains and notation we now introduce
will be used in all the examples we describe.

Let $\Gamma$ be the ideal triangle in the disk model of $\HY^2$ with vertices $A=(0,1), B=(0,-1), C=(-1,0)$ and
sides $a,b,c$ as indicated in figure  \ref{fig:figure1}.
\begin{figure}
\includegraphics{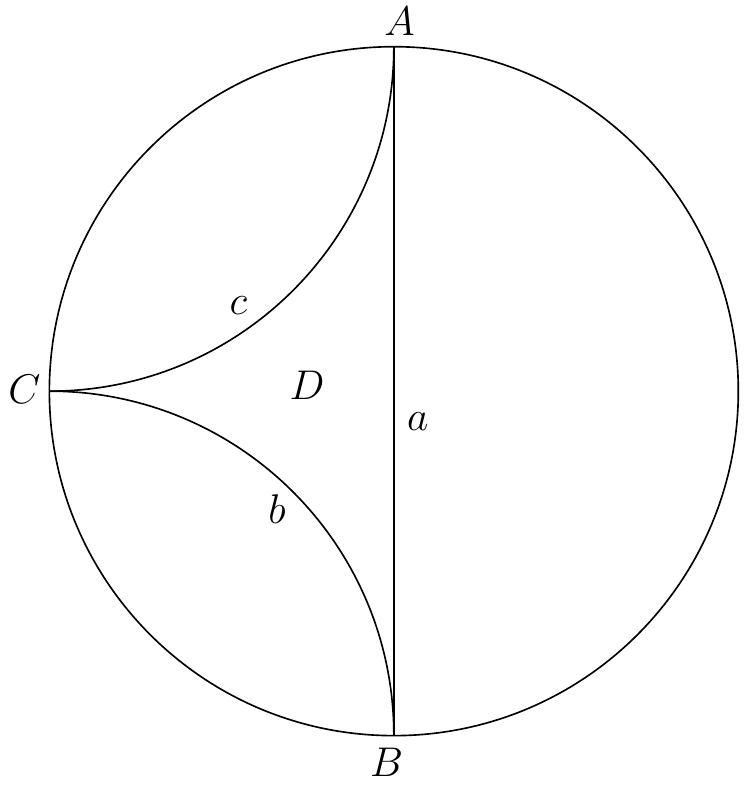}
\caption{Ideal triangle $(ABC)$ in $\HY^2$}
\label{fig:figure1}
\end{figure}

Let $\Sigma$, be the minimal graph over the domain $D$ bounded by $\Gamma$, taking the values $0$ on
$b$ and $c$ and $h>0$ on $a$. Extend $\Sigma$, to an entire minimal graph $\tilde \Sigma$ over $\HY^2$ by
rotation by $\pi$ in all the sides of $\Gamma$, and the sides of the triangles thus obtained.

In figure \ref{fig:figure2}, we indicate some of the reflected triangles and the values of the graph $\tilde \Sigma$ on their sides.
\begin{figure}
\includegraphics{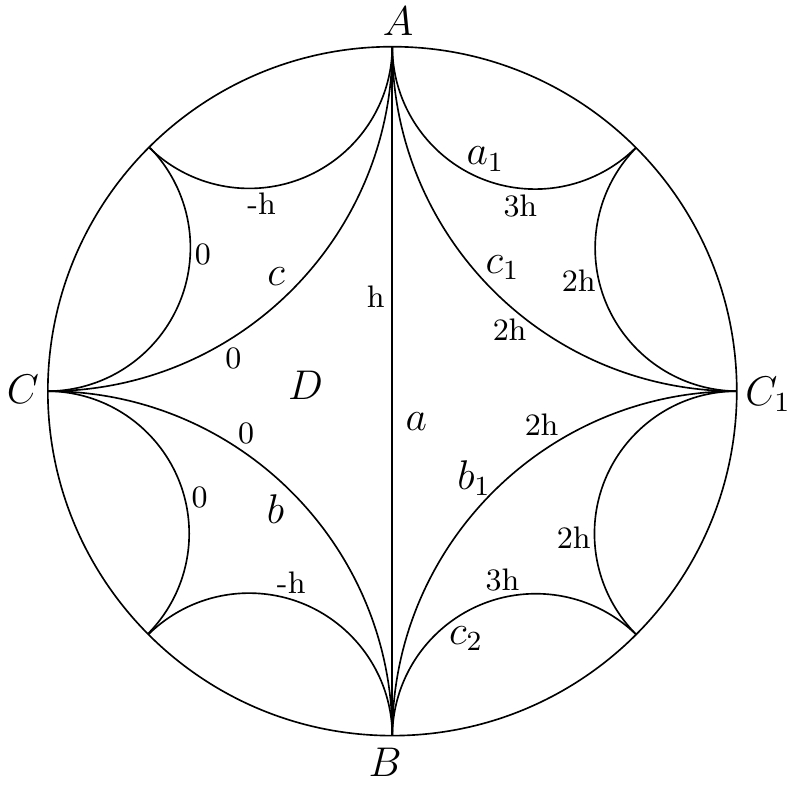}
\caption{Value of the graph $\tilde \Sigma$ on geodesics}
\label{fig:figure2}
\end{figure}

Let $D$ be the domain bounded by $\Gamma$. Let $\psi _A$ be the parabolic isometry with fixed point $A$ which takes the geodesic $c$ to $c_1$ and $a$ to $a_1$; $\psi_A=R_{c_1}R_a$, where $R_\gamma$ denotes
reflection in the geodesic $\gamma$. Let $\psi_B$ be the parabolic isometry of $\HY^2$ leaving $B$ fixed, taking $b$ to
$b_1$ and $a$ to $c_2$; $\psi_B=R_{b_1}R_a$. 

Notice that the group of isometries of $\HY^2 \times \R$, generated by $T(2h) \circ \psi_A$ and $T(2h) \circ \psi_B$,
leaves $\tilde \Sigma$ invariant.

Let $M$ be the $3$-punctured sphere obtained by identifying the sides of $D \cup R_a (D)$ by $\psi_A, \psi_B$
($c$ with $c_1$, $b$ with $b_1$). $M$ is hyperbolic and has finite area.

Let $\Sigma_2$ be the graph in $\tilde \Sigma$ over $D \cup R_a (D)$. Then the multi-graph
$\cup_{k \in \Z} \bbT _{k2h} (\Sigma_2)$ passes to the quotient $M \times (\R / T(2h))$ to give a 
complete embedded minimal surface $\Sigma$ with $3$-ends; two helicoidal and the other horizontal.
$\Sigma$ is a $3$-punctured sphere, has total curvature $-2\pi$ and $\Sigma$ is stable ($\Sigma$
is transverse to the killing field $\partial / \partial t$).

\vskip 0.2cm
{\bf Example 2.} $\Sigma$ an embedded minimal surface in $\Sph (4) \times \SY^1$, the sphere with four ends.
Let $Q$ be the ideal quadrilateral $D \cup R_a (D)$, and define $F=Q \cup R_{c_1} (Q)$.
Let $M$ be the quotient of $F$ obtained by identifying the sides of $\partial F$ as follows:

\begin{enumerate}
\item Identify $c$ with $R_{c_1} (c)$ by the parabolic isometry at $A$ taking $c$ to $R_{c_1} (c)$,
\item Identify $b$ with $R_{c_1} (b)$ by the hyperbolic isometry taking $b$ to $R_{c_1} (b)$ and,
\item Identify $b_1$ with $R_{c_1} (b_1)$ by the parabolic isometry at $C_1$ taking $b_1$ to $R_{c_1} (b_1)$.
\end{enumerate}
$M$ is a 4-punctured sphere. A more (apparently) symmetric picture of $M$ is obtained by changing the picture
by the isometry taking $A$ to $A$ and $C_1$ to $B$ as indicated in the figure \ref{fig:figure3}.

\begin{figure}
\includegraphics{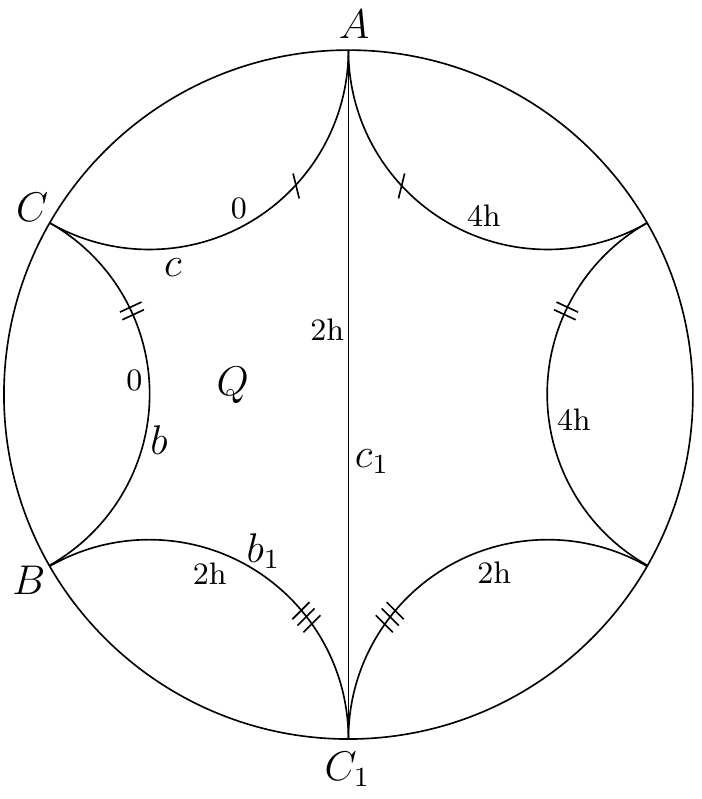}
\caption{$M$ is a 4-punctured sphere}
\label{fig:figure3}
\end{figure}

Then the graph of $\tilde \Sigma$ over $F$ yields a embedded minimal surface $\Sigma$ in 
$(M \times \R) / [T(4h)]$. $\Sigma$ has two horizontal ends and two helicoidal ends of type $E(1,1)$.
$\Sigma$ is also stable.

\vskip 0.2cm
{\bf Example 3.} A compact singly periodic Scherk surface; M a once punctured torus. This surface
is constructed in \cite{M-R-R}; we describe it here. Let $Q=D \cup R_a (D)$ and
$\gamma_1,\gamma_2$ be minimizing geodesics joining opposite sides of $D$; figure  \ref{fig:figure4}.
\begin{figure}
\includegraphics{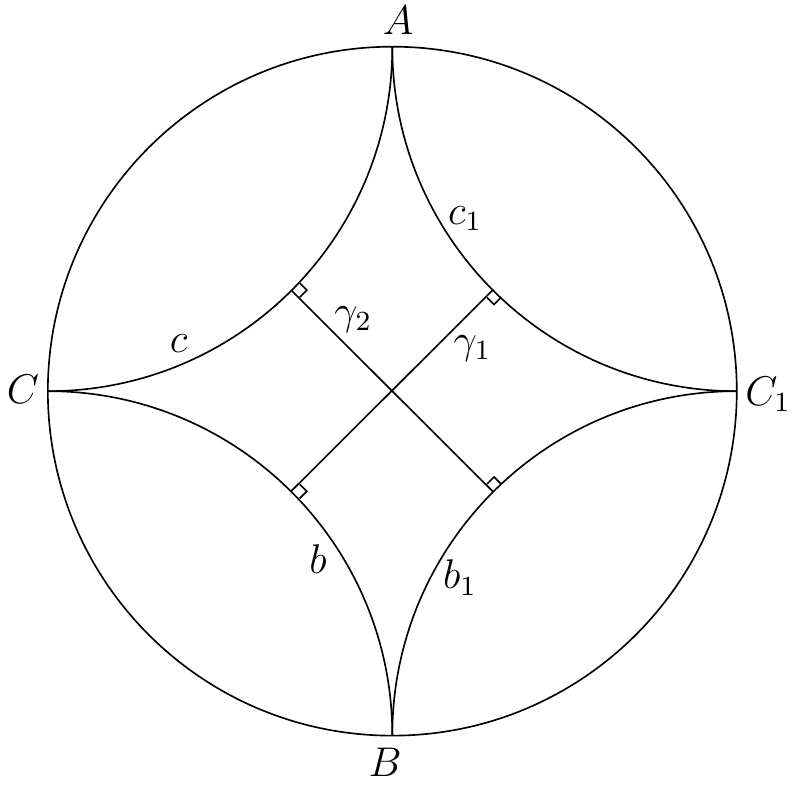}
\caption{$M$ is a once punctured torus}
\label{fig:figure4}
\end{figure} 
 In $\HY^2 \times \R$, we desingularize the intersection of the planes $\gamma_1 \times \R$ and
$\gamma_2 \times \R$ in the usual manner to create a Scherk surface invariant under a vertical translation.
We describe this. Let $\alpha$ and $\beta$ be the segments of $\gamma_1, \gamma_2$ in the first and
fourth quadrants respectively. Form a polygon in $\HY^2 \times \R$ by joining 
to $(\alpha \times \{h\}) \cup (\beta \times \{0\})$ by the two vertical segments joining 
$\alpha (\Gamma) \times \{0\}$ to  $\alpha (\Gamma) \times \{h\}$, and joining
$\beta (\Gamma) \times \{0\}$ to  $\beta (\Gamma) \times \{h\}$; $\alpha (\Gamma)$ denotes
the endpoint of $\alpha$ on $\Gamma$ (similarly for $\beta ( \Gamma)$). This polygon
bounds a least area minimal disk $D_1$; figure \ref{fig:figure5}

\begin{figure}
\includegraphics{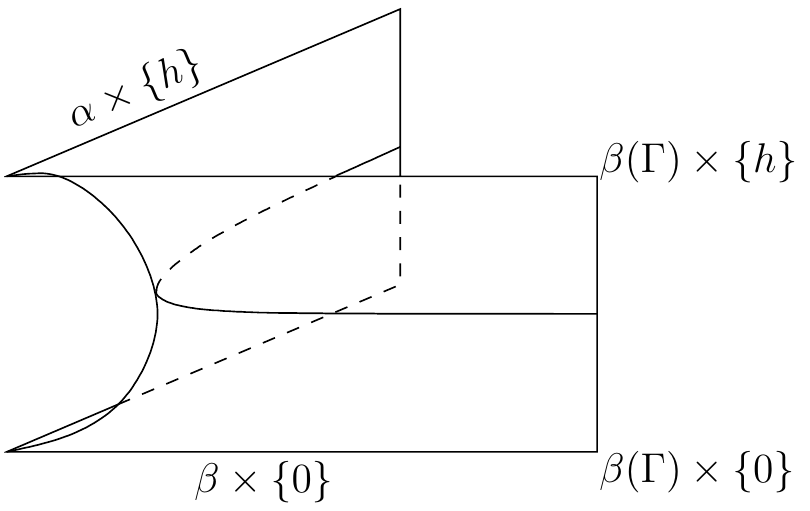}
\caption{A least area disk $D_1$}
\label{fig:figure5}
\end{figure} 

Successive symmetries in all the horizonntal sides yields a "Scherk" type surface 
in $\HY^2 \times \R$ bounded by $4$ vertical geodesics, invariant by vertical translation by $2h$.

We now identify opposite sides of $Q$ by the hyperbolic translations $T(\gamma_1), T(\gamma_2)$,
along $\gamma_1$ and $\gamma_2$. This gives a once punctured torus $M$.
The Scherk surface passes to the quotient to give a compact minimal surface $\Sigma$
with $\partial \Sigma = \emptyset$, in $M \times (\R / T(2h))$.

\vskip 0.2cm
{\bf Example 4.} A singly periodic Scherk surface with 4 vertical annular ends; $M$ a once punctured torus.
Now we "rotate" example 4 by $\pi/4$. Let $\alpha_1,\alpha_2$ be the complete geodesics joining
the opposite vertices of $\partial Q$; $\alpha_1,\alpha_2$ are the $x$ and the $y$ axis in the unit
disc model. Again we construct Plateau disks bounded by the polygon of figure  \ref{fig:figure6}.

\begin{figure}
\includegraphics{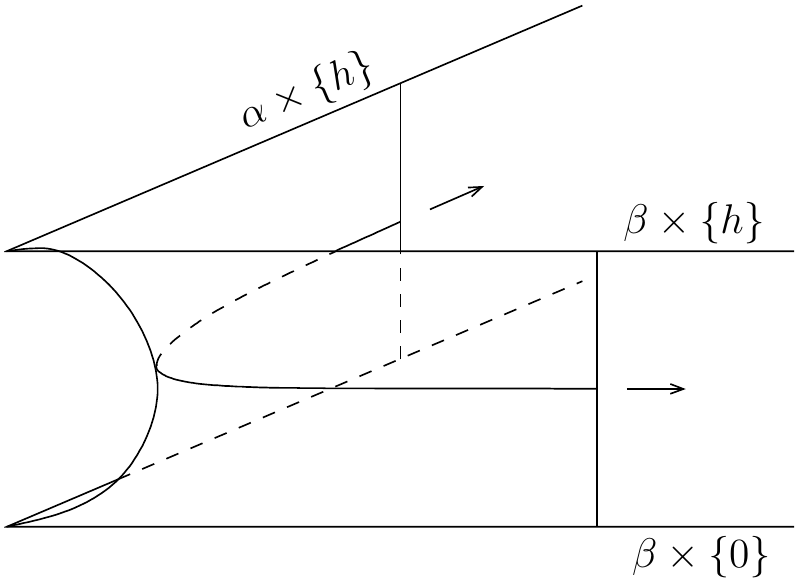}
\caption{A piece of singly periodic Scherk surface with $4$ vertical annular ends}
\label{fig:figure6}
\end{figure}

We know that when the vertical geodesic segments diverge along $\beta$, the plateau
solutions converge to a complete embedded surface in $\HY^2 \times \R$ with boundary
$\{ (x,0) / x \geq 0 \} \cup \{ (y,0) / y \geq 0\} \cup \{ (x,h) / x \geq 0\} \cup \{ (y,h) / y \geq 0 \}$.

The symmetries of this surface along all the edges yields a singly periodic Scherk surface in
$Q \times \R$, invariant by $T(2h)$.

As in example 4, we identify the opposite sides of $Q$ by hyperbolic translations to obtain a torus
$M$ with one puncture. This gives the Scherk surface $\Sigma$ in $M \times \R / [T(2h)]$, with
four vertical annular ends.

We remark that one can quotient $\partial Q$ by parabolic isometries to obtain this Scherk surface
in $M \times \SY^1$ where $M$ is now a 4-punctured sphere.

\vskip 0.2cm
{\bf Example 5.} A helicoid with  helicoidal ends in $M \times \SY^1$, $M$ a once punctured torus.
It is convenient to describe this example in $M \times \SY^1$ where $M$ is the once punctured torus
obtained from the ideal quadrilateral $Q_1$ in $\HY^2$ with the 4 vertices $(\pm \frac{1}{\sqrt 2} \pm \frac{i}{\sqrt 2})$,
by identifying opposite sides.

Let $S$ be the third quadrant of $Q_1$: $S=\{ (x,y) \in Q_1 ; x \leq 0, y \leq 0 \}$. For $h >0$,
let $\Sigma_1$ be the minimal graph over $S$ with boundary values indicated in figure  \ref{fig:figure7}.

\begin{figure}
\includegraphics{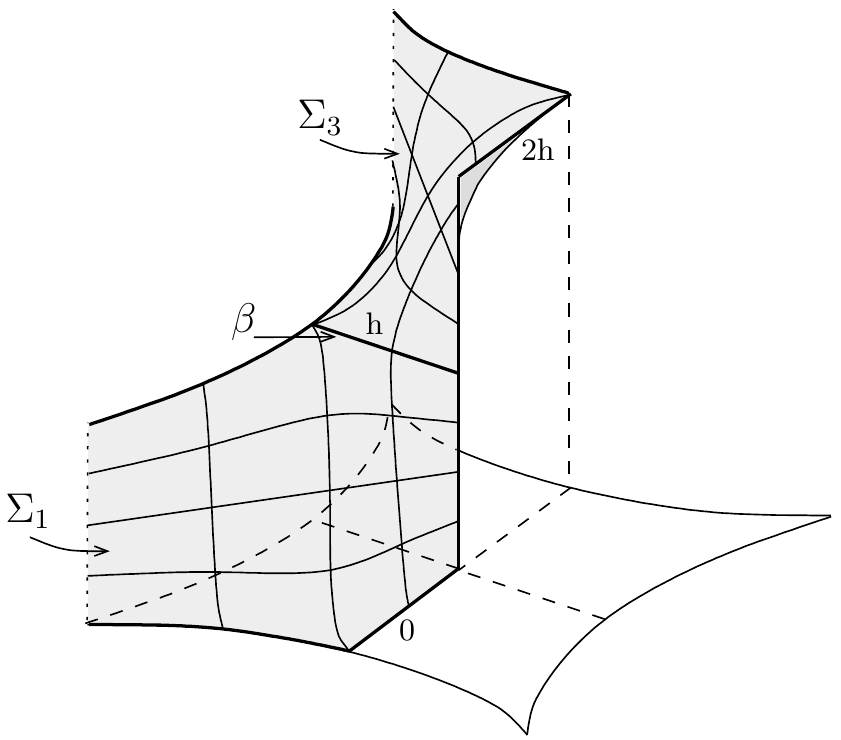}
\caption{$\Sigma_1$ be a minimal graph over $S$}
\label{fig:figure7}
\end{figure}

Let $\Sigma_3$ be the reflection of $\Sigma_1$
through $\beta$ (cf figure \ref{fig:figure7}); $\Sigma_3$ is between heights $h$ and $2h$ and is a graph over the second quadrant
of $Q_1$. Then rotate $\Sigma_1 \cup \Sigma_3$ by $\pi$ through the vertical axis between $(0,0)$ and $(0, 2h)$, to obtain
$\Sigma_2 \cup \Sigma_4$; $\Sigma_4$ is a graph over the fourth quadrant of $Q_1$. $\Sigma$ is the union of the four pieces 
$\Sigma_1$, through $\Sigma_4$, identified along the boundaries as follows.

First we consider identifying opposite sides of $Q_1$ be the hyperbolic translations sending the opposite side to the other.
Then we can quotient by $T(2h)$ or by $T(4h)$. The first quotient gives an non orientable surface in $M \times \SY^1$ with one helicoid type end. The second gives an orientable surface of total curvature $-8\pi$ with two helicoidal type
ends (it is a double cover of the first example). Topologically the first example is the connected sum of a once punctured torus and a projective plane. The second surface is 2 punctured orientable surface of genus two.

The reader can see the helicoidal structure of $\Sigma$ by going along a horizontal geodesic on $\Sigma$
at $h=0$, from one puncture to the other. Then spiral up $\Sigma$ along a helice going to the horizontal geodesic
at height $h$. Continue along this geodesic to the other (it's the same) puncture and spiral up the helices on $\Sigma$ to
height $2h$. If we do this right, we are back where we started. 

\section{Barriers in $M \times \SY^1$}
We construct barriers by solving the mean curvature equation of ruled surfaces.  These barriers will be used to prove the Trapping Theorem in section 4.   In the model
$\HY^2 \times \R =\{ (x,y,t) \in \R^{3} ; y > 0 \}$, we consider surfaces 

$$X:(u,v) \to (u, \alpha (v), v+ \lambda u)$$
for a $\calC ^2$ real  positive function of one variable $v \to \a (v)$  defined on some interval $I$.

\begin{lemma}
\label{meancurvature}
The mean curvature $H$ of the surface $X:(u,v) \to (u, \alpha (v), v+ \lambda u)$ immersed in
$\HY^2 \times \R =\{ (x,y,t) \in \R^{3} ; y > 0 \}$ with the metric $ds^2= \frac{dx^2+dy^2}{y^2}+dt^2$ is given by

 $$2H=\frac{-\a^2}{Z^3} \left( \a'' (1+\lambda^2\a^2)+\a(1+ \lambda^2 ({\a'})^{2})    \right)$$

\end{lemma}

\begin{proof} In the model of $\HY^2 \times \R =\{ (x,y,t) \in \R^{3} ; y > 0 \}$ with the metric $ds^2= \frac{dx^2+dy^2}{y^2}+dt^2$,
the non zero terms of the connection are given by
$$
\nabla_{\partial /\partial x} \partial /\partial x = \tfrac1y  \partial /\partial y  \quad \nabla_{\partial /\partial y} \partial /\partial y = -\tfrac1y \partial /\partial y $$
$$\nabla_{\partial /\partial x} \partial /\partial y = \nabla_{\partial /\partial y} \partial /\partial x= -\tfrac1y \partial /\partial x 
$$
The tangent space is generated by
\[\begin{aligned}
dX ( \partial /\partial u)&= E_{1}=\partial /\partial x + \lambda  \partial /\partial t=(1,0,\lambda) \\ 
dX ( \partial /\partial v)&= E_{2}=\a' (v)\partial /\partial y +  \partial /\partial t=(0,\a'(v),1)
 \end{aligned}\]
The direct unit normal vector  is given by $N=V/Z$ with $V= E_{1} \wedge E_{2}$,
 $$V=-\lambda \a' (v)\a(v)^{2}\partial /\partial x - \a (v)^{2} \partial /\partial y + \a' (v) \partial /\partial t =(-\lambda \a' (v) \a (v)^{2},- \a (v)^{2}, \a' (v))$$
 $$Z^{2}=|V|^{2}=(\a' (v)) ^{2} (1+ \lambda ^{2} \a (v)^2)+ \a (v)^2.$$
We compute the mean curvature by the divergence formula
$$-2H={\rm div}(N)={\rm div}\left( \frac{V}{Z} \right)=\frac{1}{Z^3} (Z^2 {\rm div} (V) - \frac12 V.(Z^2)). $$
We compute the first term

\[\begin{aligned}
 {\rm div} (V)&  =-\frac{\partial}{\partial x}(\lambda \a ^2 (v) \a ' (v)) -\lambda \a ^2 (v) \a ' (v) {\rm div} \left( \frac{\partial}{\partial x} \right)  \\
 &\;\;- \frac{\partial}{\partial y} (\alpha^2 (v))- \alpha^2 (v) {\rm div} \left( \frac{\partial}{\partial y} \right) +
 \frac{\partial}{\partial t} (\alpha ' (v)) + \alpha' (v)  {\rm div} \left( \frac{\partial}{\partial t} \right).
\end{aligned}\]
Using ${\rm div} \left( \frac{\partial}{\partial x} \right)={\rm div} \left( \frac{\partial}{\partial t} \right)=0$
and ${\rm div} \left( \frac{\partial}{\partial y} \right)= -\frac{2}{y}$ with $\alpha (v)=y$ and $v=t-\lambda x$, a direct
computation gives

\[\begin{aligned}
{\rm div} (V) & =\lambda^2 \a ^2 \a''-2 \a - \a^2 \left(-\frac{2}{\a} \right) + \a '' \\
&=(1 + \lambda^2 \a^2 ) \a''
\end{aligned}\]
For the second term

\[\begin{aligned}
\frac{1}{2} \frac{\partial}{\partial x} (Z^2)  & = -\lambda \a' \a''(1 + \lambda^2 \a^2) \\
\frac{1}{2} \frac{\partial}{\partial y} (Z^2)  & = \a (1+ \lambda^2 \a'^2)\\
\frac{1}{2} \frac{\partial}{\partial t} (Z^2)  & = \a' \a'' (1 + \lambda^2 \a^2)\\
\end{aligned}\]
Hence $\frac12 V. (Z^2)=(1 + \lambda^2 \a^2)\lambda^2\a^2\a'^2\a'' - \a^3(1+\lambda^2\a'^2)+\a'^2\a'' (1+\lambda^2\a^2).$ Finally we obtain

\[\begin{aligned}
{\rm div} (N) &={\rm div} (V/Z)\\
 & =\frac{1}{Z^3} [(1+\lambda^2 \a^2)\a''(\a^2+ \a'^2 + \lambda^2 \a^2 \a'^2)\\
 & \hspace{ 0.3cm} - (1 + \lambda^2 \a^2)\lambda^2\a^2\a'^2\a''+ \a^3(1+\lambda^2\a'^2)+ \a'^2\a''(1+\lambda^2 \a^2)] \\
&=\frac{\a^2}{Z^3} [ \a'' (1+ \lambda^2 \a^2) + \a (1+ \lambda^2 \a'^2)] = -2H
\end{aligned}\]
\end{proof}
 
We study the geometry of surfaces $X:(u,v) \to (u, \alpha (v), v+ \lambda u)$ which are minimal.
We notice they are ruled surfaces foliated by curves $v \to (0, \a (v), v)$ where $\a \in \calC ^{2} (I)$, $\a >0$, 
$\lambda \geq 0$.

\begin{figure}
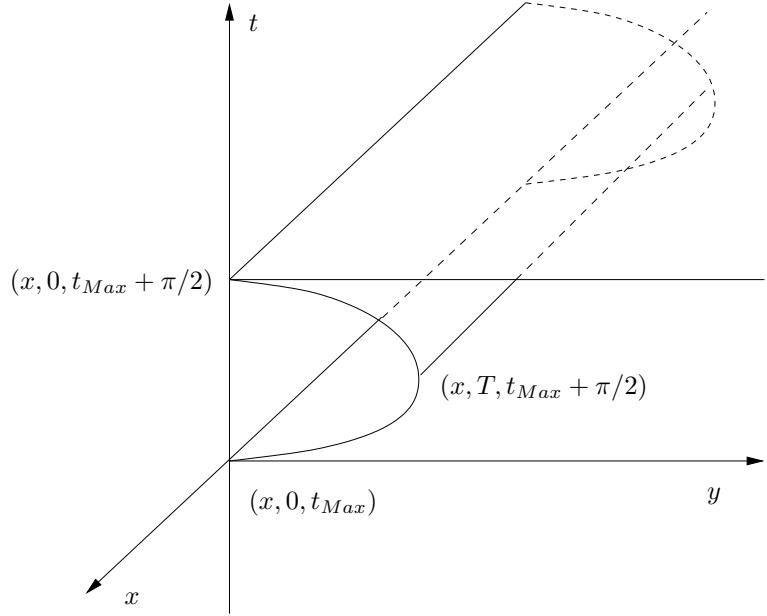
\caption{$S^0$, Surfaces foliated by horizontal horocycles}
\label{fig8}
\end{figure}

The first case solves the equation when $\lambda=0$. The solution $\a (v) = T \sin v$ gives
the family of minimal surfaces up to vertical translation
$$S^{0}_T=\{ (u, T \sin v, v) \in \R^3 ; u \in \R, v \in [0,\pi]\}$$
This surface is foliated by horizontal horocycles $u \to (u,\a (v),v)$ and is described in 
Hauswirth \cite{Hau}, then by Toubiana and
Sa Earp \cite{SaEarp-Tou}, Daniel \cite{daniel} and Mazet, Rodriguez, Rosenberg \cite{M-R-R}(see figure \ref{fig8}).
By the nature of the curve $v \to (0, T \sin v, v)$, the surfaces $S^0_T$ foliate the slab  $S=\{ (x,y,t) \in \R^3;   0< y, 0 \leq t \leq \pi\}$.

\begin{figure}
\resizebox{5cm}{!}{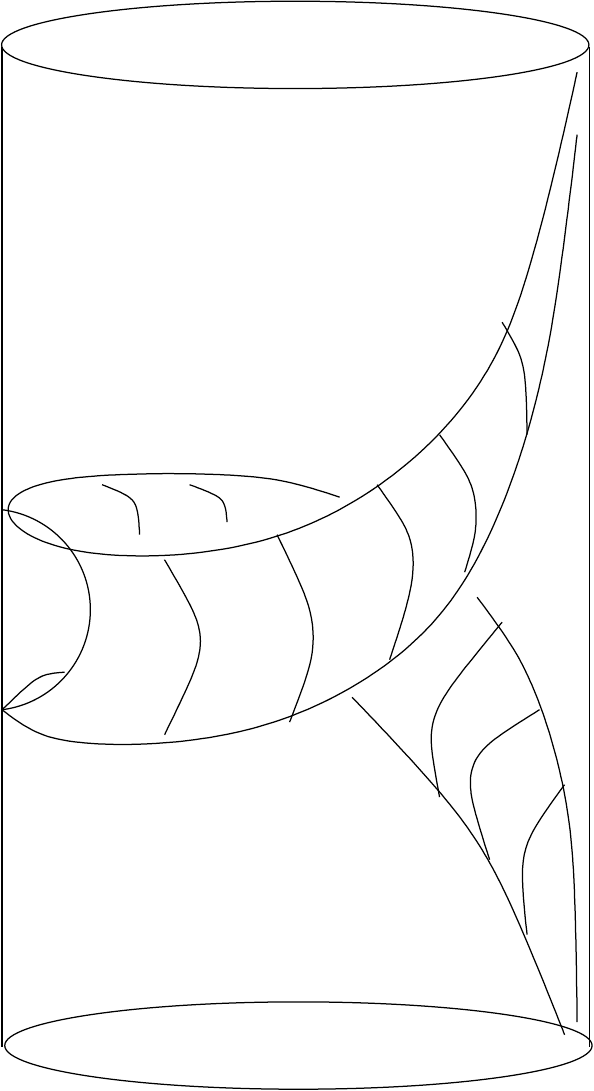}\hfill
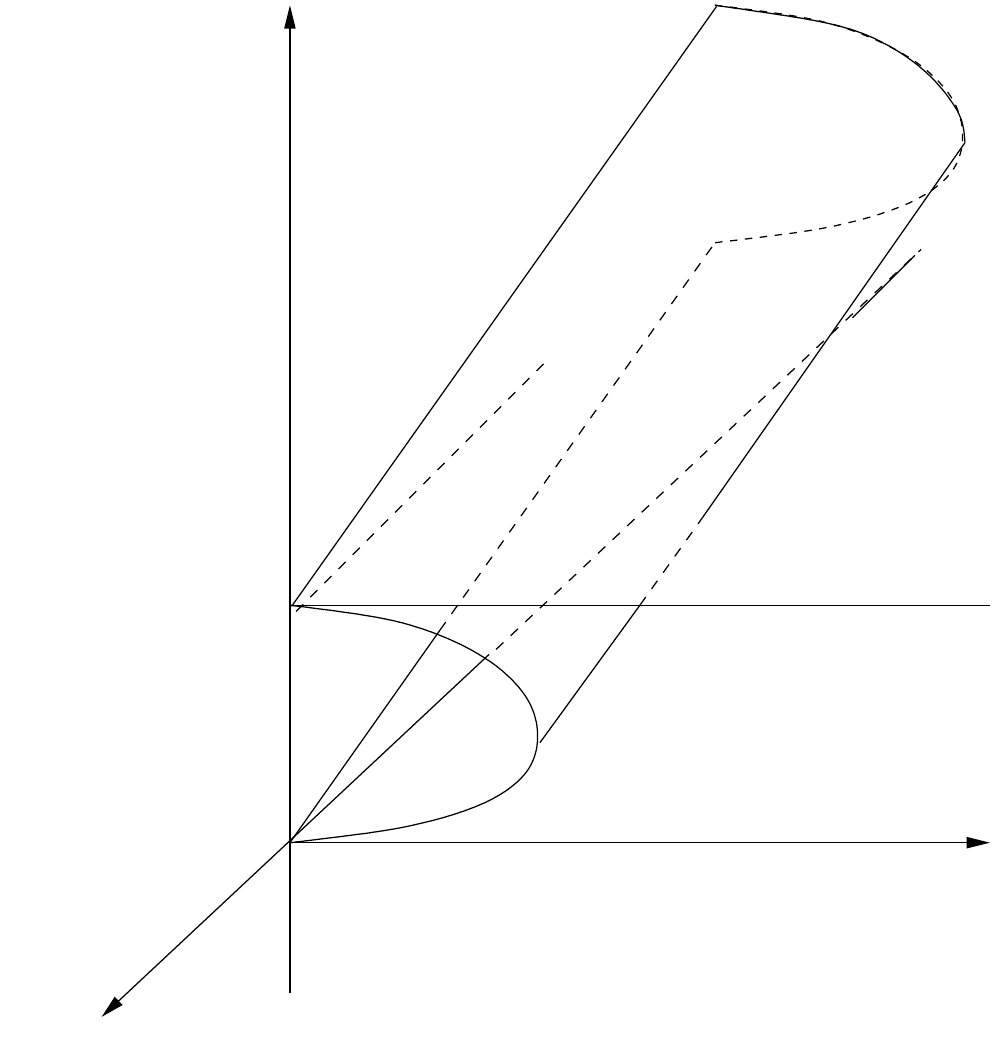
\caption{Surfaces $S^\lambda_T$ in the disk model and half-space model of $\HY \times \R$}
\label{fig9}
\end{figure}

The general case $\lambda \neq 0$ depends on a function $\a$ , a solution of the equation
$$\a'' (1+\lambda^2\a^2)+\a(1+ \lambda^2 (\a')^2)=0.$$
This equation has a first integral $(1+ \lambda^2 \a '^2)(1+\lambda^2 \a^2)=T$ for some fixed constant $T >1$. Since
$\a'' <0$ the curves $v \to (0, \a (v), v)$ are convex. For fixed $T>1$, the function $\a_T (v)$ has its maximum value
at  $\sup \a_T (v)=\lambda^{-1}\sqrt{T-1}$. The function $\a_T$ is positive on a set $[0,v_0 (T)]$.
The solution $\a_T(v)$ with initial data $\a_T(0)=0$ and
$\a' (0)=\lambda ^{-1}\sqrt{T-1}$ defines a one-parameter family of 
minimal surfaces 
$$S^{\lambda}_T=\{ (u, \a_T (v) , v+ \lambda u) \in \R^3 ; u \in \R, v \in [0,v_0 (T)]\}.$$
For large values of $T$ and fixed constant $M >0$, we look for the set of values where $0 \leq \a_T(v) \leq M$.
On this set we remark that 
$$\a'^2=\lambda^{-2}\left( \frac{T}{1+ \lambda^2 \a^2} -1 \right) \geq \lambda^{-2}\left( \frac{T}{1+ \lambda^2 M^2} -1 \right) $$
which implies that $\a' _T (v) \to \infty$ when $T \to \infty$. The part of the curve $(0, \a_T (v), v)$ contained in $ 0<y \leq M$ converges to the half geodesic $\{ (0,y,0)\in \R^3 ; 0 < y \leq M\}$. 

We summarize this discussion in the figure \ref{fig9}  and we will use
the following lemma:

\begin{lemma} 
\label{barrier}
a) The family of surfaces $S^0_T$, foliates the slab $\HY^{2} \times [0,\pi]$ and when  $T$ goes to infinity the surfaces $S^0_T$ converge on compact sets to the horizontal section $\HY^2 \times \{ 0\}$.

b) The one-parameter family of surfaces $S^\lambda_T$ converges on compact sets
to $\{(x,y,t) \in \R^3; y >0 \hbox{ and } t= \lambda x\}$.
\end{lemma}

\section{Trapping theorem for minimal ends}
We consider a minimal surface $\Sigma$ of finite topology, hence each end of $\Sigma$ is an annular end.
Since $\Sigma$ is properly immersed, each end $A_0$ of $\Sigma$ is contained in some end $\calM$ of $M \times \SY^1$.

\begin{lemma}
\label{proper}
There is $y_0 \geq 1$ and a sub-end $A$ of $A_0$ such that $\partial A \subset \bbT (y_0)$,
$A$ is transverse to $\bbT (y_0)$ and $A \subset \cup_{y \geq y_0} \bbT (y)$.
\end{lemma}

\begin{figure}
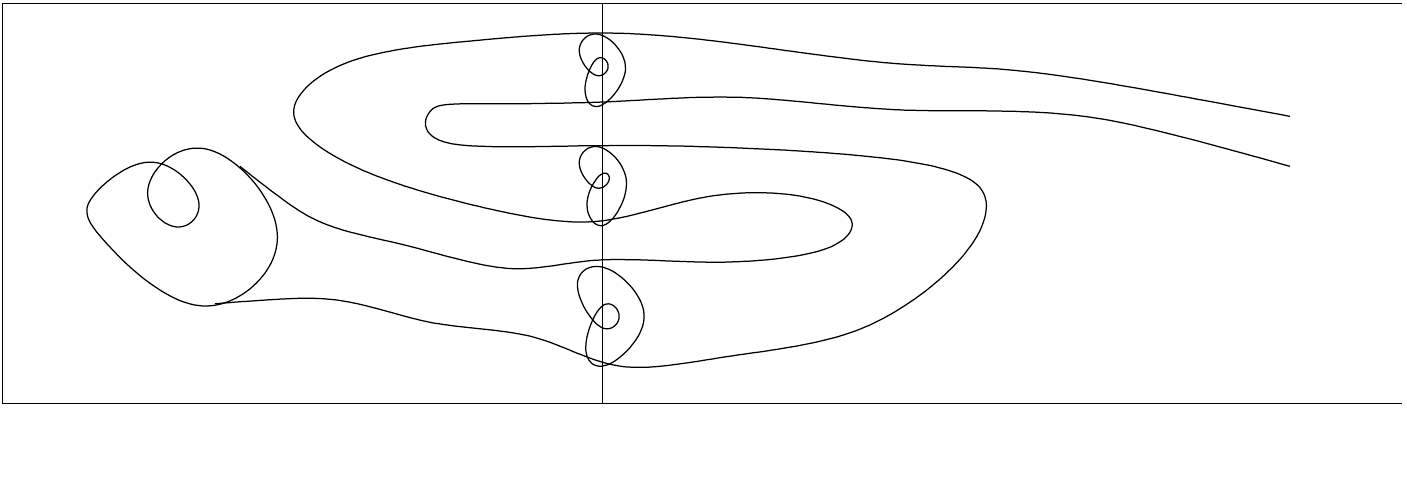
\caption{An annular end in $\calM$}
\label{fig10}
\end{figure}

\begin{proof}
Since $\Sigma$ is properly immersed each end of $\Sigma$ has a subend $A_0$ contained in some
$\calM= \cup_{y \geq 1} \bbT (y)$. $A_0$ is transverse to almost every $ \bbT (y)$ so let
$y_0 >1$ be such that $\partial A_0 \subset  \cup_{1\leq y < y_0} \bbT (y)$, and $A_0$ is transverse
to $ \bbT (y_0)$. Then $A_0 \cap \bbT (y_0)= C_1 \cup ...\cup C_k$, each $C_j$ an immersed Jordan curve in $\bbT (y_0)$.

$A_0$ is proper so $A_0 \cap \bbT (y) \neq \emptyset$ for large $y$. Hence at least one of the $C_j$ is not null homotopic
in $A_0$. Observe that there is at most one such $C_j$. For if $C_i$ and $C_j$ are not trivial then they bound a compact domain
$F$ in $A_0$ disjoint from $\partial A_0$. $F$ cannot be contained in $\cup_{1\leq y \leq  y_0} \bbT (y)$ since then $F$ would
touch some $\bbT(y_1), y_1 < y_0$, on the mean convex side of $\bbT(y_1)$, a contradiction. So $F \subset \cup_{y \geq y_0} \bbT (y)$.
But then, $\partial A_0$ and $C_i$ or $C_j$ ($C_i$ say) would bound a compact $F_1$ on $A_0$, $F_1 \cap C_j = \emptyset$.
$A_0 - F_1$ is an annular sub-end of $A_0$ with boundary $C_i$ contained in $\partial \bbT (y_0)$. Since $A_0 - F_1$ intersects
$\bbT (y_0)$ also at $C_j$, there is a compact domain $F_2$ of $A_0 -F_1$ contained in $\cup_{1\leq y \leq y_0} \bbT (y)$
with $\partial F_2 \subset \bbT (y_0)$; a contradiction; see figure \ref{fig10}.

Now it is clear that if each $C_\ell$, $\ell \neq i$ bounds a disk $D$ on $A_0$ that is contained in $\cup_{y \geq y_0} \bbT (y)$.
It follows that the connected component of $A$ in  $\cup_{y \geq y_0} \bbT (y)$ that has $C_i$ in its boundary has no other
$C_\ell$, $l \neq i$, in its boundary. This proves the lemma.
\end{proof}
 
By a change of coordinates on $\calM$, we can assume that the end $A$ is in $\cup_{y \geq 1} \bbT (y)$ and $\partial A \subset \bbT (1)$ . Let $E$ be a connected
component of the lift of $A$ to $\HY \times \R$. The boundary $\partial E \subset P:=\{(x,y,t) \in \R^3 ; y=1 \}$
and $E$ is transverse to $P$. There is $(p,q)$ such that the curve $\partial E$ is invariant by the isometry of $\HY^2 \times \R$

$$\psi^p \circ T(h)^q : (x,y,t) \to (x+p\tau , y, t+qh).$$ 
We say that $A$ and $E$ are of type $(p,q)$. The curve $\partial A$ is a curve of the torus $\bbT (1)$.   We prove in the following lemma that $(p,q) \neq (0,0)$

\begin{lemma}
\label{annulus}
The end $E$ is topologically a half-plane and $\partial E$ is a non compact curve in $P$.
\end{lemma}

\begin{figure}
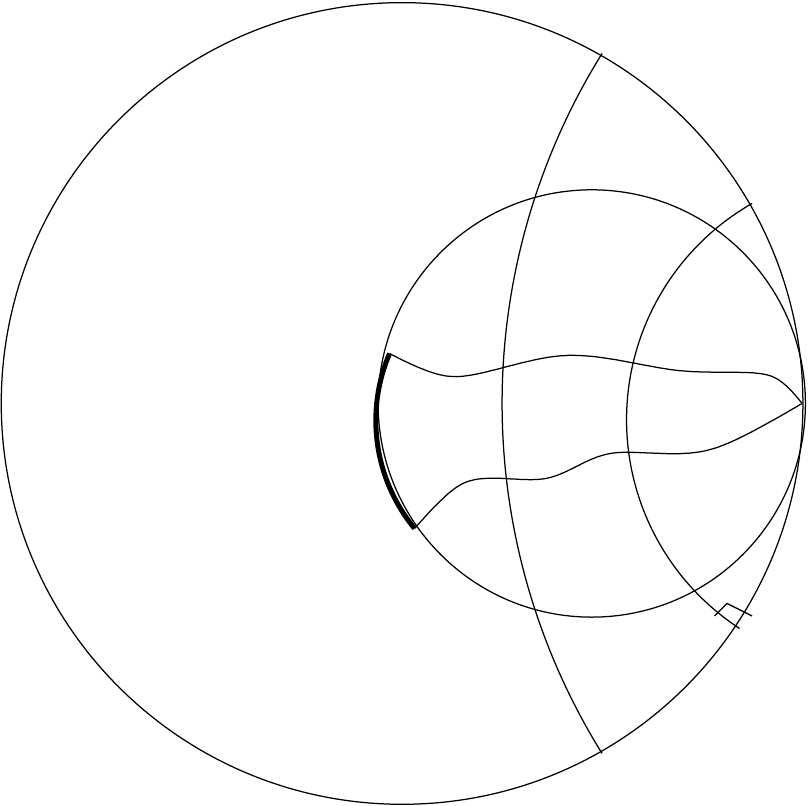
\caption{An annular end in $\calM$}
\label{fig11}
\end{figure}

\begin{proof}
Assume the contrary, and let $E$ be a lifting of $A$ to $\HY \times \R$, $E$ an immersed annulus
in $\{ y \geq 1\}$, $\partial E \subset H=\{ y=1\}$. We know the coordinate $y$ is a proper
function on $E$.

Denote by $\Pi : \HY \times \R \to \HY= \HY \times \{ 0 \}$, the vertical projection. Let
$\gamma_1, \gamma_2$ be disjoint geodesics of $\HY$, disjoint from $\Pi (\partial E)$,
and that separate $\Pi (\partial E)$ from the point at infinity of $\HY$; see figure \ref{fig11}
(the set $\Pi (\partial E)$ is compact). Let $\Omega \subset \HY$ be the domain of $\HY$ bounded
by $\gamma_1 \cup \gamma_2$, so $\Pi(\partial E) \cap \Omega= \emptyset$.

For $a \in \R$, solve the Dirichlet problem on $\Omega$ to find a minimal graph over $\Omega$,
with asymptotic values $+\infty$ on $\gamma_1 \cup \gamma_2$, and $a$ on $\partial_\infty (\Omega) $
(see Collin-Rosenberg \cite{collinrosenberg}).

By varying $a$ we obtain a first point of contact of the graph with $E$; a contradiction.
\end{proof}
Now we prove that an end $E$ of type $(p,q)$, $(p,q) \neq (0,0)$ is trapped between two ends of type $E_{(p,q)}$:
 \begin{theorem}{\bf  (The Trapping Theorem)}
 \label{trapping}
 Let $A \subset \Sigma$ be a properly immersed end in $\calM$ with $\partial A \subset \bbT (1)$ and
 $A$ transverse to $ \bbT (1)$. If $\partial A $, is a curve of type $(p,q)$ in $\bbT (1)$, then $A$ is contained in a slab
 on $\calM$ bounded by two standard ends $A_{(p,q)}$.
\end{theorem}

\begin{proof}  We use the model of $\calM=(H \times \R) / [\psi, T(h)]$ and a connected component $E$ of a lifting of $A$ in $\HY^2 \times \R$. We prove that  $E$ is contained in a slab bounded by two
half-planes $E_{(p,q)}$ in $H \times \R$ where $H=\{ (x,y) \in \R^2; y \geq 1\}$. 
\vskip 0.3cm
\begin{figure}
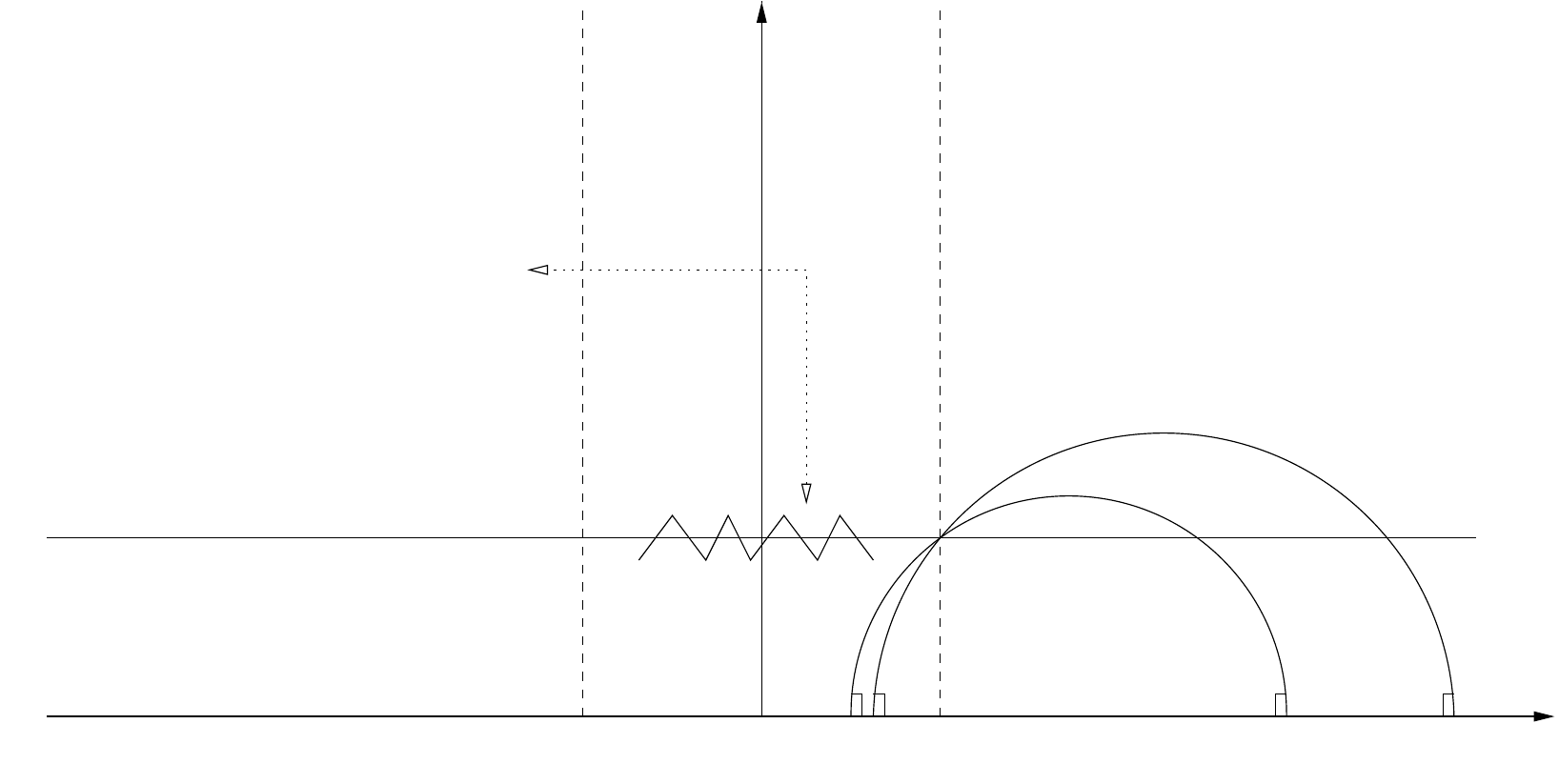
\caption{An annular end in $\calM$}
\label{fig12}
\end{figure}
{\it Case $(p,q)=(0,q)$}.  First we begin with the case where the curve $\partial A$ is of type $(0,q)$. This means that the boundary $\partial E$ is a periodic curve invariant by vertical translation $T(h)^q=T(qh)$. From this invariance of $\partial E$, we know there exists $x_{Min}$ and
$x_{Max}$ such that $\partial E \subset \{ (x,1,t); x_{Min}\leq x \leq x_{Max} \}$.

Let $Q=\{ (x,y) \in \R^2 ; \; x \geq x_{\rm Max}   \hbox{ and }  y > 0 \}$. Foliate $Q$ by the geodesics $\gamma_T$ whose end-points at
infinity are $(x_{Max},0)$ and $(T,0)$; $T > x_{Max}$. For $| T - x_{Max}| < 2$, $\gamma_T \cap  \{ y \geq 1\} = \emptyset$. Define
$S_T= \gamma_T \times \R$; so $S_T \cap E= \emptyset$ for $|T -x_{Max}| < 2$.

Now let $T$ increase to $\infty$, so $S_T$ converges to $\{ (x_{Max},y) ; y>0 \} \times \R$. Since $E$ is periodic, $S_T$ must be disjoint
from $E$ for all $T>1$; otherwise there would be a first point of contact (i.e. the two surfaces cannot have a first contact 
point at infinity), contradicting the maximum principle (see figure \ref{fig12})

The same argument using $x_{Min}$ shows $E$ is trapped between two standard ends of type $E_{(0,1)}$.
\vskip 0.3cm
{\it Case $(p,q)=(p,0)$} Now $\partial E$ is a curve invariant by $\psi^p (x,y,t)=(x +p\tau,y,t)$. Let $t_{Min}$ and $t_{Max}$
satisfy:
$$\partial E \subset \{ (x,1,t) ; t_{Min} \leq t \leq t_{Max} \}.$$
Translate the barriers of lemma \ref{barrier}, see figure \ref{fig8},
$$S^{0}_T (t):=\{ (u, T \sin v, v+t) \in \R^3 ; u \in \R, v \in [0,\pi]\} \hbox{ with } t \geq t_{\rm Max}.$$
For $|T| \leq 1$, we have $S^{0}_T (t) \cap ( H \times \R)=\emptyset$, and $\partial E$ is below height $t= t_{Max}$. By lemma \ref{barrier},
the family $S^{0}_T (t_{Max})$ converges on compact sets to the horizontal section $t=t_{Max}$. For $t > t_{Max}$, $t$ large, $T_0$ given,
we have $S^{0}_{T_0} (t) \cap E = \emptyset$.

If $S^{0}_{T_0} (t_{Max}) \cap E \neq \emptyset$, then since $E$ is periodic, there would be a first $t_1$ such that $S^{0}_{T_0} (t_1) \cap
E \neq \emptyset$, contradicting the maximum principle. Thus $E$ is below $t=t_{Max}$. The same argument with
$$\tilde S^{0}_T (t):=\{ (u, T \sin v, v-\pi+t) \in \R^3 ; u \in \R, v \in [0,\pi]\}  \hbox{ with } t \leq t_{\rm Min}$$
shows $E$ is above $t =t_{Min}$. Thus $E$ is trapped between two standard ends of type $E_{(p,0)}$.

\vskip 0.3cm
{\it Case $(p,q)\neq (0,q),(p,0)$}. Now we use the family of barriers $S^{\lambda} _T$. $\partial E$ is invariant by 
the isometry $\psi^p \circ T(h)^q : (x,y,t) \to (x+p\tau, y, t +qh)$ on $y=1$. Thus there exists  $c_{\rm Min}, c_{\rm Max}$
such that
$$\partial E \subset \{ (x,1,t)\in \R^3; c_{\rm Min}\leq p\tau t - qhx \leq c_{\rm Max} \}.$$
We use $S^{\lambda} _T$ of lemma \ref{barrier} with $\lambda = \frac{qh}{p\tau}$; see figure \ref{fig9}.
$$S^{\lambda}_T (t):=\{ (u, \alpha_T (v), v+\lambda u+t) \in \R^3 ; u \in \R, v \in [0,v_0(T)]\} \hbox{ with } t \geq c_{\rm Max}/( p\tau)$$
For $T>1$ fixed, there is $t_0 >  c_{\rm Max}/( p\tau)$ large so that $S^{\lambda}_T (t_0) \cap E = \emptyset$. Decreasing
$t$ from $t_0$ to $c_{\rm Max}/( p\tau)$ we conclude (there is no first point of contact) that $E$ is below
$S^{\lambda}_T (c_{\rm Max}/( p\tau)$ for any $T$. Let $T \to \infty$; the $S^{\lambda}_T (c_{\rm Max}/( p\tau)$ converge to
$\{ (x,y,t)\in \R^3;   p\tau t - qhx = c_{\rm Max} \}$, hence 
$$E \subset \{ (x,y,t) \in \R^3;  p\tau t - qhx \leq c_{\rm Max} \}.$$

The same argument with
$$\tilde S^{\lambda}_T (t):=\{ (u, \alpha_T (v), v+\lambda u- v_0 (T)+t) \in \R^3 ; u \in \R, v \in [0,v_0(T)]\} \hbox{ with } t \leq c_{\rm Min}/ (p\tau)$$
shows that 
$$E \subset \{ (x,y,t) \in \R^3;  p\tau t - qhx \geq c_{\rm Min} \},$$
which completes the proof of the theorem.
\end{proof}

\section{The Dragging Lemma}

\begin{lemmad} Let $g:\Sigma \to N$ be a properly immersed minimal surface in a complete
$3$-manifold $N$. Let $A$ be a compact surface (perhaps with boundary)
and $f: A \times [0,1] \to N$ a $\calC^1$-map such that $f(A \times \{ t \})=A(t)$ is a minimal
immersion for $0 \leq t \leq 1$. If $\partial(A(t)) \cap g(\Sigma) = \emptyset$ for $0\leq t \leq 1$
and $A(0) \cap g(\Sigma) \neq \emptyset $, then there is a ${\calC}^1$ path $\gamma (t)$ in $ \Sigma$, such that $g \circ \gamma(t) \in A(t) \cap g( \Sigma) $ for $0 \leq t \leq 1$. Moreover we can prescribe any initial value $g \circ \gamma (0) \in A(0) \cap g(\Sigma)$.
\end{lemmad}

\begin{remark}
To obtain a $\gamma (t)$ satisfying the Dragging lemma that is continuous (not necessarily $\calC ^1$)
it suffices to read the following proof up to (and including) Claim 1.
\end{remark}

\begin{proof}
When there is no chance of confusion we will identify in the following $\Sigma$ and its image $g(\Sigma)$,
$\gamma \subset \Sigma$ and $g\circ \gamma$ in $g (\Sigma) \subset N$. In particular when we consider embeddings
of $\Sigma$ there is no confusion. 

Let $\Sigma (t)= g(\Sigma) \cap A(t)$ and $\Gamma (t)= f^{-1} (\Sigma (t))$, $0 \leq t \leq 1$ the pre-image in $A \times [0,1]$.

When $g: \Sigma \to N$ is an immersion, we consider $p_0 \in g(\Sigma) \cap A(0)$,
and pre-images $z_0 \in g^{-1} (p_0)$ and $(q_0,0) \in f^{-1} (p_0)$.  We will obtain the arc $\gamma (t) \in \Sigma$ in a neighborhood of $z_0$ by a lift of an arc $\eta (t)$ in a neighborhood of  $(q_0,0)$ in $\Gamma ([0,1])$ i.e. $g\circ \gamma (t)=f \circ \eta(t)$.
We will extend the arc continuously by iterating the construction.

Since $\Gamma (t)$ represents the intersection of two compact minimal surfaces, we know $\Gamma (t)$ is a set of a finite number of  compact analytic curves $\Gamma_1 (t),...,\Gamma_k (t)$. These curves $\Gamma_i (t)$ are analytic immersions of topological circles. By hypothesis, $\Gamma (t) \cap (\partial A \times [0,1])= \emptyset$ for all $t$. The maximum principle assures that  the immersed curves can not contain a small loop, nor an isolated point. Since $A(t)$ is compact and has bounded curvature, a small loop in $\Gamma (t)$ would bound a small disc $D$ in $\Sigma$ with boundary in $A$. Since $A$ is locally a stable surface, we can consider
a local foliation around the disc and find a contradicttion with the maximum principle.
We say in the following that $\Gamma (t)$ does not contain small loops.

{\bf Claim 1:} We will see that for each $t$ with $\Gamma (t) \neq \emptyset$, $t<1$ there is a $\delta (t) >0$ such that if $(q,t) \in \Gamma (t)$, then there is a $\calC ^{1}$ arc $\eta(\tau)$ defined for $t  \leq  \tau \leq t + \delta (t)$ such that $\eta (t) =(q,t)$ and $\eta(\tau) \in \Gamma (\tau)$ for all $\tau$ (there may be values of $t$
where $\gamma ' (t)=0$).

Since $\Gamma (0) \neq \emptyset$, this will show that the set of $t$ for which $\eta (t)$ is defined is a non empty open set. This defines an arc $\gamma (\tau)$ as a lift of  $f \circ \eta (\tau) \subset A(\tau)$ in a neighborhood of $\gamma (t) \in \Sigma$.

\begin{figure}
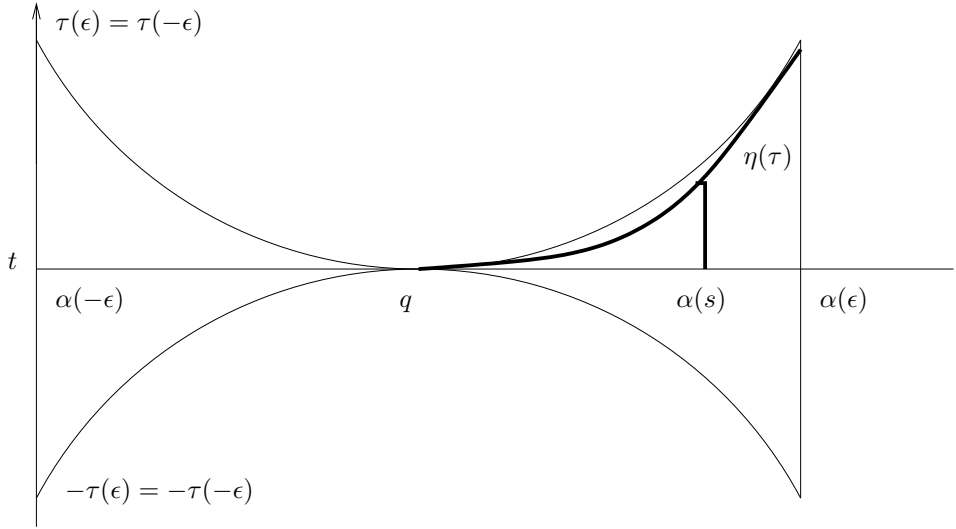
\caption{Neighborhood of a singular point }
\label{fig:figure14d}
\end{figure}
First suppose $(q,t)  \in \Gamma (t)$ is a point where  $A(t)=f (A \times \{ t\})$ and $g(\Sigma)$ are transverse
at $f(q,t)$. Let us consider the ${\calC}^1$ immersions
$$F: A \times [0,1] \to N \times [0,1] \hbox{ with } F(q,t)= (f(q,t),t)$$
$$G: \Sigma \times [0,1] \to N \times [0,1] \hbox{ with } G(z,t)= (g(z),t).$$
Let $\hat M=F(A \times [0,1]) \cap G (\Sigma \times [0,1])$ and $M=F^{-1} (\hat M)$ .  $F(A \times [0,1])$ and $G(\Sigma \times [0,1])$ are transverse at 
$p=F(q,t)$. Thus $\hat M$ is a 2-dimensional surface of $N \times [0,1]$ near $p$ .  We consider $X (t)$ a tangent vector field along $\Gamma (t)$ and $JX(t)$ an orthogonal vector field to $X(t)$ in $T_{(q,t)} M$. 
If $\partial / \partial t \perp T_p \hat M$, then $T_p \hat M = T_{f(q,t)} A(t)=T_{f(q,t)} g(\Sigma)$ and 
$(q,t)$ would be a non transverse point of intersection of $A(t)$ and $g(\Sigma)$. 
Thus $<JX(t), \partial / \partial t > \neq 0$ and we can find $\eta (\tau)$  a smooth path, 
defined for $\tau \in [t-\delta (q), t + \delta (q) ]$ such that $\eta (t)=(q,t)$ and $\eta ' (t)=JX(t)$ is transverse to $\Gamma (t)$ at $(q,t)$.

By transversality and $f$ being $\calC ^1$ in the variable $t$, we have  a $\delta (q)>0$ such that for 
$t - \delta (q)  \leq \tau \leq t + \delta (q)$, $A(\tau)$ intersects $f \circ \eta (\tau)$ in a unique point and this point 
varies continuously with $t- \delta (q)  \leq  \tau \leq t + \delta (q)$.
With a fixed initial point in $\Sigma$, a lift of $f \circ \eta (\tau)$, defines
$\gamma (\tau) \in \Sigma$.

Again by transversality, we can find a neighborhood of $(q,t)$ in $\Gamma (t)$
and a $\delta >0 $ so that the above path $\gamma (\tau)$ exists for $t -\delta \leq \tau \leq t + \delta$, through each point in the neighborhood of $q$. It suffices, to look for a local immersion of a neighborhood of $0$ in $T_p M$ into $M$, to obtain a ${\calC}^1$ diffeomorphism $\psi : B(0) \subset T_pM \to M$. $M$ has the structure of a ${\calC}^1$ manifold in a neighborhood of points of transversality and this structure extends to $F^{-1} (M) \subset A\times [0,1]$.

\begin{figure}
\label{fig:figured2}
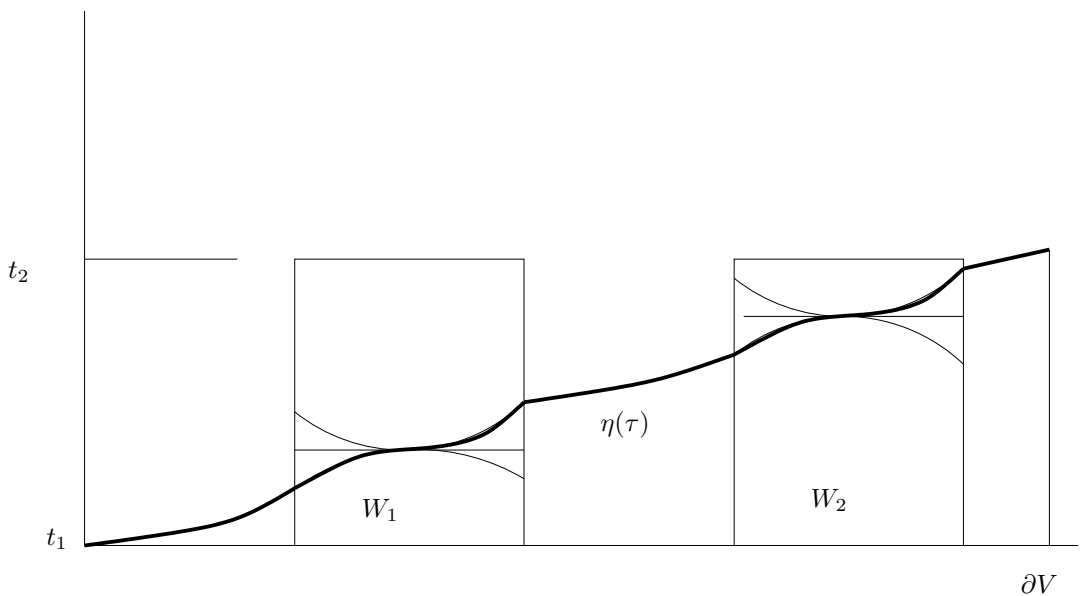
\caption{The curve $\eta (\tau)$ passing through several singularities. }
\end{figure}

We  will find a $\delta >0$ that works in a neighborhood of a singular point $(q,t) \in \Gamma (t)$, where there is a $z \in \Sigma$ such that $f(q,t)=g(z)$ and  $T_{f(q,t)} A(t)=T_{g(z)} g(\Sigma)$. 
We consider singularities of $\Gamma (t)$ where $A(t)$ and $g(\Sigma)$ are tangent. Near a singularity  
$(q,t) \in \Gamma (t)$, $\Gamma (t)$ contains $2k$ analytic curves intersecting at $q$ at equal angles, $k\geq 1$.

Let $V$ be a neighborhood of $q$ in $A$. The set $\Gamma (t) \cap V$ is $2k$ analytic curves. Let $\alpha: ] -\epsilon,\epsilon [ \to V\cap \Gamma (t)$ be a regular parametrization of one curve 
with $\alpha (0) =q$ and $\alpha (\pm \epsilon) \in \partial V$. By transversality as discussed in the previous paragraph $<JX(t), \partial / \partial t > \neq 0$
at $\alpha (s)$ for $s \neq 0$ and $JX (t)$ can be integrated as a curve on $M$ for $t- \delta (s)  \leq \tau \leq t + \delta (s)$.
Here $\delta (s)$ is a ${\calC}^1$ function which can be chosen increasing with $\delta (0)=\delta' (0)=0$.

There exists  a ${\calC}^1$ diffeomorphism  $\phi: \Omega= \{ (s,\tau) \in \R^2; -\epsilon \leq s \leq \epsilon , t - \delta (s)  \leq  \tau \leq t+  \delta (s)\} \to M$ such that $\phi (s,t) = \alpha (s)$ for $s \in ]-\epsilon, \epsilon [$ and $\phi (s ,\tau) \in \Gamma (\tau)$ for  
$t- \delta (s)  \leq \tau \leq t + \delta (s)$. We consider a function $\tau: ]-\epsilon,\epsilon[ \to \R$, such that $(s, \pm \tau (s)) \in \Omega$ and $\tau$ is increasing, $\tau (0)=\tau' (0)=0$ and $\tau (\epsilon)=t+\delta ( \epsilon )$,  $\tau (-\epsilon)=t+\delta (- \epsilon )$.

Now we can construct a path $ \eta (\tau) \in \Gamma (\tau)$ which joins $(q,t)$ to a point in $\Gamma(t+\delta (\epsilon))$. The ${\calC}^1$ arc $f \circ \eta (\tau), t \leq \tau \leq t+\delta (\epsilon)$ is locally parametrized by $\phi (s, \tau (s)), s \in ]0 ,\epsilon[$ and continuously extends to $f(q,t)$ when $\tau \to t$.
Each point $ \alpha (s)$, can be connected ${\calC}^1$, by the arc
$\phi (s, \tau), t \leq \tau \leq \tau (s)$ from $\alpha (s)$ to  $\phi (s, \tau (s))$, and next a subarc of $\eta (\tau)$ for $\tau (s) \leq \tau \leq t+ \delta (\epsilon)$ (see figure \ref{fig:figure14d}). The constant $\delta (\epsilon)$ depends only on $\alpha (\epsilon)=q_1$, and we note $\delta (q_1):=\delta (\epsilon)$.

Now there are a finite number of arcs $\alpha$ in $V-(q)$, with end points $q$ and a collection
of $q_1,q_2,...q_{2k}$. So one has a $0< \delta $ with $\delta < 
\delta (q_i)$ that works in a neighborhood of $q$. The claim is proved.

To complete the proof of the Dragging Lemma, it suffices to prove that $\gamma (t)$ extends $\calC^1$ for any value of $t \in [0,1]$. Assume that there is a point $t_0$ such that the arc $\gamma (t)$ is defined in a $\calC ^1$ manner for $t < t_0$.
By compactness of $A$, the arc accumulates at a point $(q,t_0) \in \Gamma (t_0)$. Remark that the structure
of $M$ along $\Gamma (t_0)$ gives easily the existence of a continuous extension to $t_0$. To ensure
a ${\calC}^1$ path through $t_0$, we need a more careful analysis at $(q,t_0)$.

{\bf Claim 2:} Suppose the path $\gamma (t)$ satisfies the conditions of the Dragging lemma for
$0 \leq t \leq t_0 <1$. Then $\gamma (t)$ can be extended to $0< t< t_0 + \delta$, to be $\calC^1$ and
satisfy the conditions of the Dragging lemma, for some $\delta >0$.

If $(q,t_0)$ is a transversal point, $M$ has a structure of a manifold and if $t_0-\delta (t_0) < t_1 < t_0$ and $\eta (t_1)=(q_1, t_1)$ is in a neighborhood of $(q,t_0)$, we can find a $\calC ^1$ arc that joins $\eta (t_1)$ to $(q,t_0) \in \Gamma (t_0)$.
Next we extend the arc for $t_0 \leq t \leq t_0 + \delta (t_0)$.

If $(q,t_0)$ is a singular point, we consider a neighborhood $V \subset A$ of $q$ and $\Gamma (t_0)$ intersects $\partial V$ in $2k$
transversal points $q_1,...,q_{2k}$.  We consider $V \times [t_1,t_0]$ with $t_0 - \delta (t_0) <t_1 <t_0$. By transversality at $(q_1,t_0),...,(q_{2k},t_0)$, the analytic set $\Gamma (t_1)$ intersects $\partial V$ in $2k$ points and $V$ in $k$ analytical arcs $\alpha_1,..., \alpha _k$. We suppose that $\eta (t_1) \in \alpha_1 \subset  V \times \{t_1\}$. We construct below a monotonous $\calC ^1$ arc from $\eta (t_1)$ to a point $(\hat q,t_2)$ on $\partial V \times \{t_2\}$ for some $t_1 < t_2 < t_0$ and by transversality
an arc from $(\hat q,t_2)$ to a point $(q',t_0) \in \partial V \times \{t_0\}$, using the fact that $ t_0 - \delta (t_0) < t_2$. Next we can extend the arc in a $\calC^1$ manner from $(q',t_0)$ to some point in $\Gamma (t_0 + \delta (t_0))$.

\begin{figure}
\label{fig:figured3d}
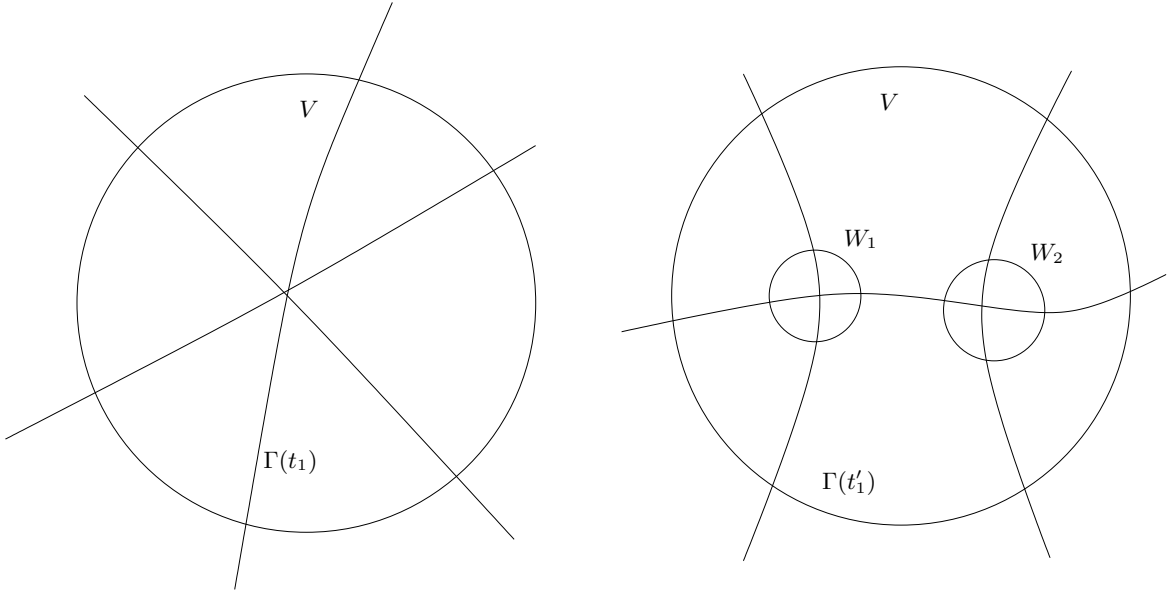
\caption{Left: The curve $\Gamma (t_1)$-Right: The curve $\Gamma (t_1')$.}
\end{figure}

We consider $(\tilde q_1, t_1),...,(\tilde q_\ell, t_1)$ singular points of $\Gamma (t_1) \cap V\times \{t_1\}$ and we denote by $W_1,...,W_\ell$ neighborhoods of $\tilde q_1,...,\tilde q_\ell$ in $A \cap V$. The arc $\alpha_1$ cannot have double points in $V$ without creating small loops. Hence $\alpha_1$ passes through  each $W_1,...,W_\ell$ at most one time, before  joining a point of $\partial V$ (We can restrict $V$ in such a way that there are no 
small loops in $V$). 

First we assume that there is $t_2$ such that for any $t \in [t_1,t_2]$, the curve $\Gamma (t)$ has exactly one isolated singularity in each neighborhood $W_i \times \{t \}$ with the same type
as $\tilde q_i \in \Gamma (t_1)$ ($i=1,...,\ell$) and $t_2< t_1 + \delta (t_1)$. 
If we parametrize $\alpha_1 : [s_0,s_{2\ell+1}]Ê\to \Gamma (t_1)$, we can find $s_1,...,s_{2\ell}$ such that $\alpha_1 (s_{2k-1}), \alpha_1 (s_{2k}) \in  \partial W_k$ and $I_k =[s_{2k-2}, s_{2k-1}]$ are intervals parametrizing transversal points in $\Gamma (t_1)$.

The manifold structure of $M$ gives an immersion $\psi _j: I_j \times [t_1 , t_1 + \delta ] \to M$, $t_1+ \delta < t_2$ and $j=1,...,\ell+1$. In the construction of $\eta$ up to $t_1$, the singular points are isolated;
 then we can assume $\eta (t_1)$ is a regular point of $\Gamma (t_1)$, hence is contained in an
 $\alpha_1 (I_j)$. We construct the beginning of the arc $\eta (\tau)$ as the graph parametrized by
 $\phi_j (s , \tau (s))$ with $\tau$ an increasing function from $t_1$ to $t_1 + \delta/n$ as $s$ varies
 from $\hat s \in I_j$, corresponding to the initial point $\eta (t_1)= \alpha_1 (\hat s)$, to $s_{2j-1}$. Next
 we pass through the singularity $(\tilde q_j, t_1 + 2 \delta/n)$ by constructing
an arc wich joins the point $\phi_j (s_{2j-1},t_1+\delta/n) \in \Gamma (t_1 + \delta/n) \cap \partial W_j$ to
the point $\phi _{j+1}(s_{2j}, t_1 + 3\delta/n) \in \Gamma (t_1 + 3\delta/n) \cap \partial W_j$ (see figure 14).
For a suitable value of $n$ we can iterate this construction, passing through the singularities
$\tilde q_j, \tilde q_{j+1}...$, until we join a point $(\hat q, t_2)$ of $\partial V \times \{ t_2 \}$ and then
we extend the arc up to $t_0$ by transversality outside $V$.

Now we look for this interval $[t_1,t_2]$. Let $t_1<t'_1<t_0$ and $\Gamma (t'_1)$ have several singularities in
some neighborhood $W_k$, or a unique singularity of index less the one of the $\tilde q_k$. We consider in this $W_k$ a finite collection of neighborhoods of isolated singularities $W'_{k,1},...W'_{k,\ell'}$.  We observe, by transversality that there are the same number of components of $\Gamma (t_1)$ and $\Gamma (t'_1)$ in $W_k$  (see figure 15). Hence each $W'_{k,j}$ contains a number of components of  $\Gamma (t'_1)$ strictly less than the number of components of $\Gamma (t_1)$ in $W_k$. The index of the singularity is strictly decreasing along this procedure. We can iterate this analysis up to a point
where each singularity can not be reduced to a simple one. This gives the interval $[t_1,t_2]$.

\end{proof}


\section{Compact minimal annuli}
\label{annuli}

\begin{figure}
\resizebox{10cm}{!}{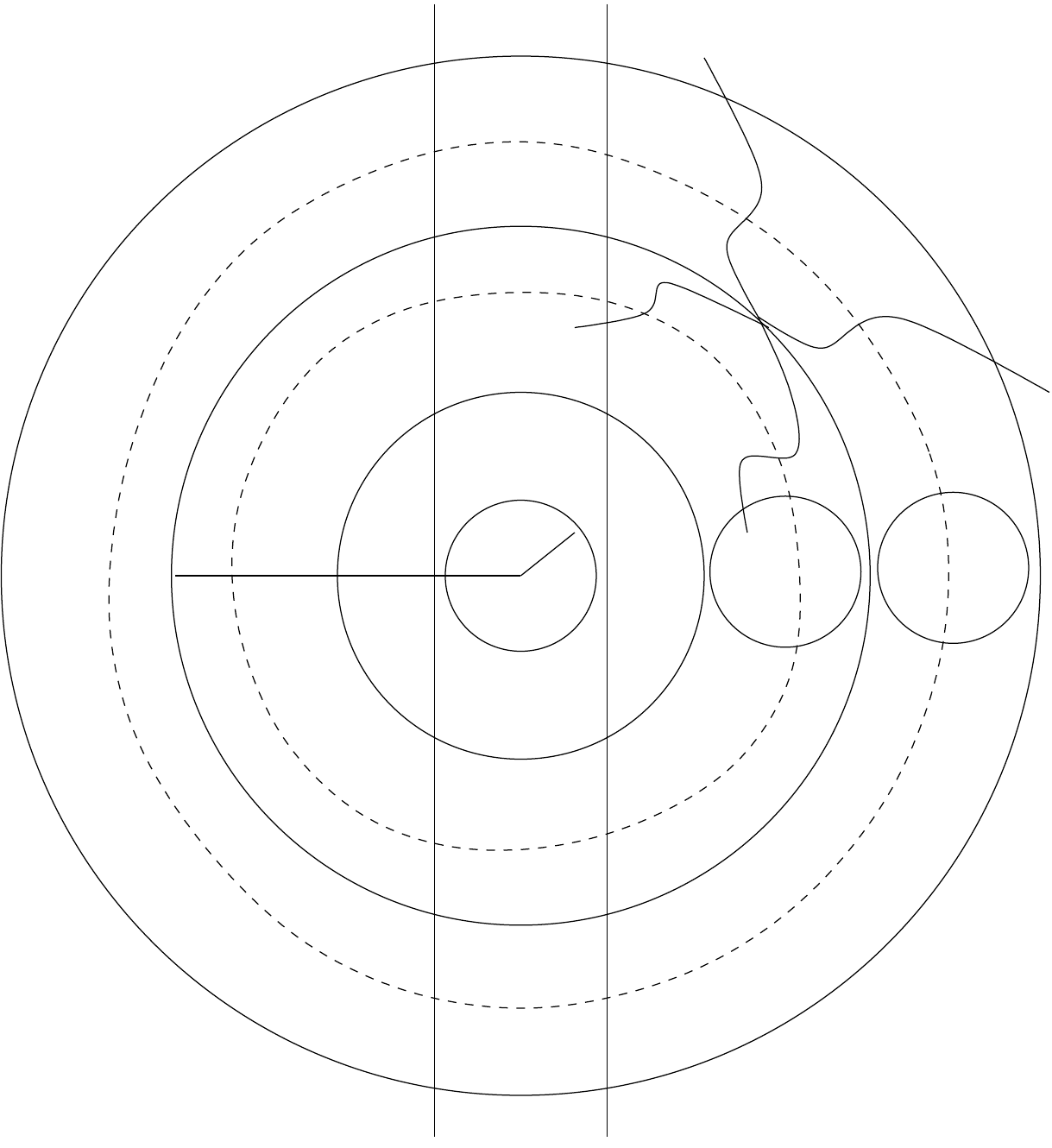}
\caption{ }
\label{fig13}
\end{figure}

We now introduce the compact stable horizontal minimal annulus $F_0$ bounded by circles in vertical planes
$P(c)$ and $P(-c)$ where $P(c)=\{ (x,y,t) \in \R^3; y >0 \hbox{ and }x=c\}$. We also foliate a tubular neighborhood ${\rm Tub} (F_0)$
of $F_0$ by compact minimal annuli $F_s$, $-1 \leq s \leq 1$ and certain
small balls $B_\rho$ containing horizontal minimal annuli ${\calC}_{\ell}$.

We will now construct the stable compact annulus $F_0$. Let $\eta$ be
a circle of radius one in the vertical plane $P(c)$, centered at $(c,y,0)$. The metric
induced on $P(c)$ is euclidean so circles make sense. As $y \to \infty$,
${\rm dist} (\eta, P(0)) \to 0$. The disk of least area in $\HY^2 \times \R$
bounded by $\eta$ is the disc in $P(c)$ bounded by $\eta$ (by the maximum principle).
The area of this disk does not depend $y$. So for $y$ large, there is a compact annulus
in $\HY^2 \times \R$ with one boundary in $P(0)$ and the other boundary $\eta$, whose area is less that the area of
the disk in $P(c)$ bounded by $\eta$. Assume $y$ is large. then by the Douglas criterium there is a least area annulus
$F_+$ having one boundary $\eta$ and the other in $P(0)$. Since $F_+$ has least area w.r.t. this boundary condition,
$\partial F_+ \cap P(0)$ is orthogonal to $P(0)$. Hence $F_0$, the symmetry of $F_+$ through $P(0)$, union $F_+$, is a smooth compact minimal  annulus orthogonal to $P(0)$ and $F_+ \cap P(0)$ is convex. The normal vector along this curve  takes on all directions in the plane $P(0)$.
\begin{figure}
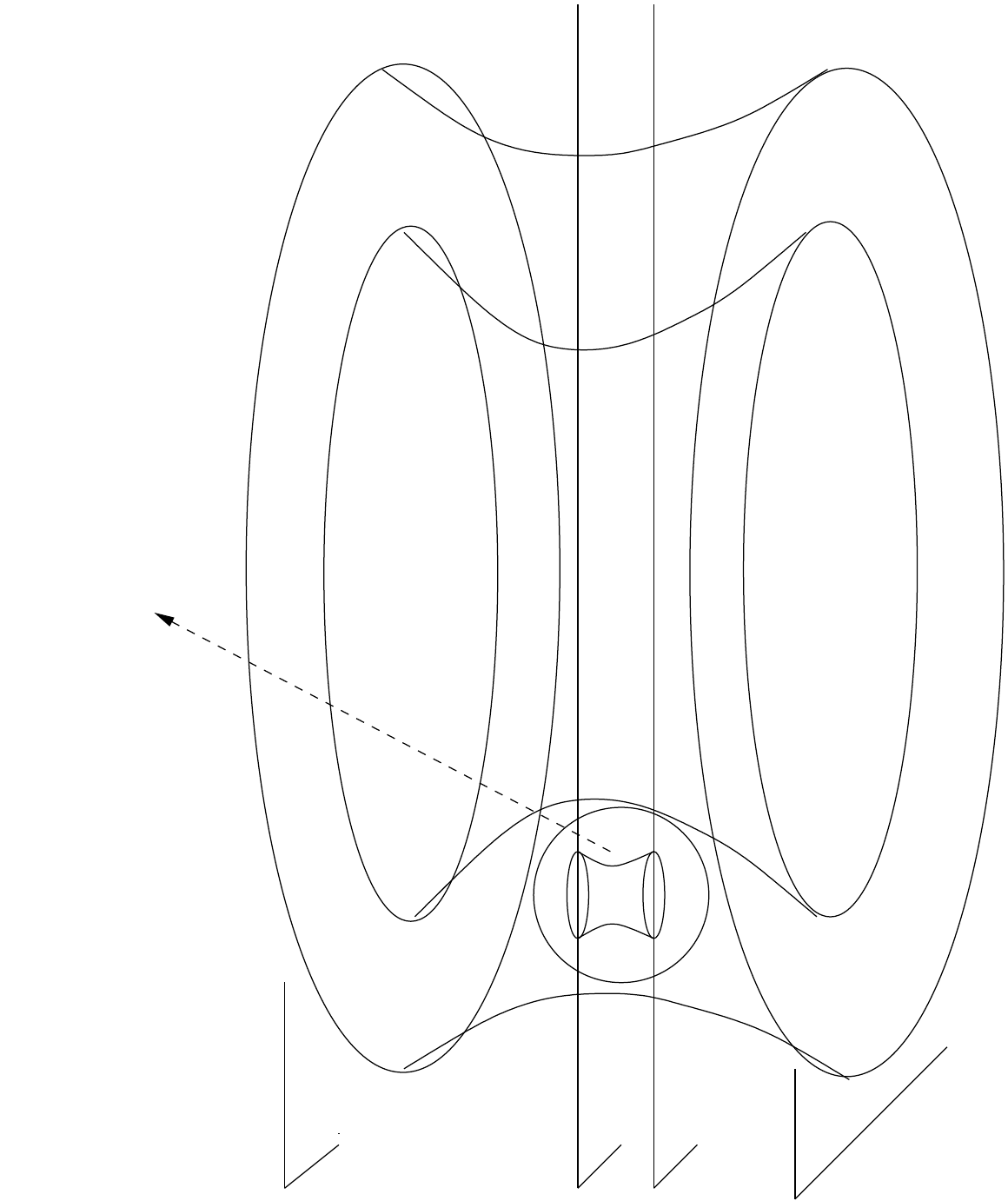
\caption{ }
\label{fig14}
\end{figure}
Let $\sigma$ be symmetry through $P(0)$, $\eta_-=\sigma (\eta)$, $F_-= \sigma (F_+)$. Observe that $F_0$ has least
area with boundary $\eta \cup \eta_-$. For if $B$ is an annulus with $\partial B= \eta \cup \eta_-$, write
$B=B_+ \cup B_-$ where $B_+ = \{ (x,y,t) \in \R^3; y >0 \hbox{ and } 0 \leq x \leq c \} \cap B$, and $B_- = \{ (x,y,t) \in \R^3; y >0 \hbox{ and } -c \leq x \leq 0 \} \cap B$. We know that the ${\rm Area} (B_+) = |B_+| \geq \frac{|F_0|}{2}$ and $ |B_-| \geq \frac{|F_0|}{2}$ so $|B| \geq |F_0|$. Thus $F_0$ is a stable annulus as desired. Let $\gamma_1$ be the geodesic joining 
$(c,y,0)$ to $(-c,y,0)$. We assume $y$ large so that $\gamma_1 \cap F_0 = \emptyset$.

Let $\eta (s)$ be equidistant circles of $\eta$ in $P(c)$, for $|s|$ small, $\eta (0)=\eta$. Let $\eta_- (s)= \sigma (\eta (s))$
be equidistant circles in $P(-c)$, $\eta_- (0)= \eta_-$. Since $F_0$ is strictly stable, there is a $\delta >0$ so that for
$|s| \leq \delta$, there is a foliation of a tubular neighborhood ${\rm Tub} (F_0)$, of $F_0$
by compact minimal annuli $F(s)$, with $\partial F(s)= \eta (s) \cup \eta_- (s)$. choose $\delta$ sufficiently small so that
$${\rm dist}({\rm Tub}(F_0), \gamma_1) >0.$$

Let ${\rm Slab} (c)=\{ (x,y,t) \in \R^3; y >0 \hbox{ and } |x| \leq c \}$.
We denote by $F_0^-$ the  bounded component of   ${\rm Slab} (c)-F_0$, and by $F_0^+$ the other component of
 ${\rm Slab} (c)-F_0$. The annuli $F_s$ are inside $F_0^-$ for $s \in [-1,0]$ and inside $F_0^+$ for $s \in [0,1]$.

We consider ${\rm Tub}^-(F_0)= \cup _{s \in [-1, 0]} F_s$ and ${\rm Tub}^+(F_0)= \cup _{s \in [0, 1]} F_s$;  domains of $\HY^2 \times \R$. 

We consider the curves $S_+=P(0) \cap {F}_{1/2}$ and $S_-=P(0) \cap {F}_{-1/2}$. There exists
a constant $\rho >0$, such that for any $q$ of $S_+$( or $S_-$) the geodesic ball $B_\rho (q)$ of geodesic radius $\rho$ centered at $q$ is contained in ${\rm Tub}^+(F_0)$ (resp. ${\rm Tub}^-(F_0)$).

We can find $\ell >0$ such that any geodesic ball of radius $\rho$ centered at $q$ contains a small compact minimal annulus 
$\calC _\ell$ bounded by two geodesic circles contained in $P(\ell) \cap B_\rho$ and $P(-\ell)\cap B_\rho$.
We say in the following that $\calC _\ell$ is centered at $q \in S_+ \cup S_-$. We denote by $q_1$ the point $\gamma_1 \cap P(0)$; see figure \ref{fig14}.

In summary, we fix $\rho >0$ such that

\begin{enumerate}
\item $3 \rho < {\rm dist}(F_0, \gamma_1) $,
\item $B_{2 \rho} (q_1) \cap {\rm Tub }^- (F_0) = \emptyset $
\item $B_\rho (q) \subset {\rm Tub }^- (F_0)$ for any $q \in S_-$,
\item $B_\rho (q) \subset {\rm Tub }^+ (F_0)$ for any $q \in S_+$,
\end{enumerate}
Now we clearly have the following:

{\bf Claim:} 
  Any continuous curve $\gamma$ in the interior of ${\rm Tub}^+(F_0) \cap {\rm Slab} (\ell)$ (or  ${\rm Tub}^-(F_0)\cap {\rm Slab} (\ell)$) joining $F_0$ to $F_{1}$ (or $F_{-1}$) intersects a compact annulus of the family $C_\ell (q) \subset {\rm Tub}^+(F_0)$(resp. ${\rm Tub}^-(F_0)$) for some point $q \in S_+$ (reps. $q \in S_-$).
\begin{figure}
\includegraphics{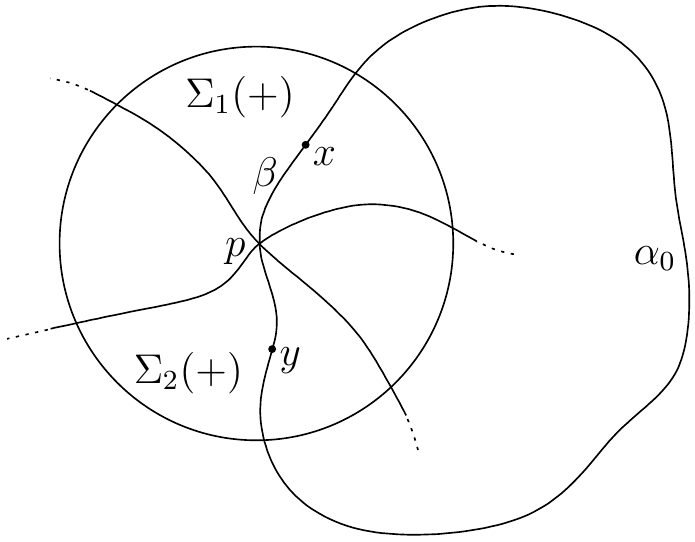}
\caption{ }
\label{fig14bis}
\end{figure}
The next proposition gives at least two components of $\Sigma$ in $F_0^-$ when $F_0$ is tangent to $\Sigma$ at some point.
\begin{proposition}
\label{componentcatenoid}
Let $\Sigma$ be a properly immersed minimal half-plane in ${\rm Slab}(\ell)$. Suppose $\Sigma$ is tangent to
$F_0$ at $p$ and $\partial \Sigma \cap F_0 = \emptyset$. Then there are at least two connected
 components  of $\Sigma$ in $F_0^-$. More precisely if $\Sigma_1 (-)$ and $\Sigma_2 (-)$ are distinct
 local components of $\Sigma$ in $F_0^-$ then $\Sigma_1 (-)$ and $\Sigma_2 (-)$ are in distinct
 components of $\Sigma \cap F_0^-$.
\end{proposition}

\begin{proof}
If not we can find a path $\alpha_0$ in $\Sigma \cap F_0^-$, joining a point $x \in \Sigma_1 (-)$ and $y \in \Sigma_2 (-)$. Then join
$x$ to $y$ by a local path $\beta_0$ in $\Sigma$ going through $p$, but $\beta_0 \subset F_0^-$ except at $p$ 
(see figure \ref{fig14bis}).
Let $\Gamma = \alpha_0 \cup \beta_0 \subset F_0^-$. Since $\Sigma$ is a half-plane, $\Gamma$ bounds a disk $D$ in
 $\Sigma$.  By construction $D$ contains points in the interior of $F_0 ^+$.

Hence there is a compact  component of $D$  in $F_0^+$ with boundary in $F_0$. By the maximum principle
$D \cap F_t \neq \emptyset$, $0 \leq t \leq 1$ and there is at least one point $p_1$ of $D \cap F_{1}$. Using compact annuli
${\calC}_\ell$ inside ${\rm Tub}^+(F_0)$, we can find an annulus ${\calC}_\ell (q)$  which intersects $D$ (by the claim). 
Now translate this catenoid in the interior $F_0^+$ to a point outside the convex hull of $F_0$. Apply the Dragging lemma 
to obtain points of $D$ outside the convex hull. This contradicts the maximum principle.
\end{proof}
\section{A family of graph barriers}
\label{graphebarriere}
In this section we study a one parameter family of surfaces $\Sigma_n$ graphs on a sequence of domains $\Omega_n$ of $\HY^2$bounded by two geodesics.  In the unit disk model of $\HY^2= \{ (x,y)\in \R^2; x^2+y^2 < 1 \}$,
we consider two geodesics $\gamma_n$ and $\gamma_{-n}$ passing through the points $(-1+1/n,0)$ and $(1-1/n,0)$) and both orthogonal
to $\{ y=0\}$. We consider the domain $\Omega_n$ bounded by $\gamma_n$ and $\gamma_{-n}$ (see figure \ref{fig16}, Left). We solve the minimal
graph equation for a function  $u_n:\Omega_n \to \R$ with $u_n= + \infty$ on $\gamma_n \cup \gamma_{-n}$ and $u_n =0$ on 
$\partial_{\infty} \Omega_n$, the boundary at infinity of $\Omega_n$.

The graph $u_n$ has a line of curvature $\Gamma_n$ over the geodesic $\gamma_0=\{ (x,y) \in \HY^2; x=0\}$. The following proposition describes the limit of the graphs $\Sigma_n$ when $n \to \infty$.

\begin{proposition}
\label{graphesigma}
The sequence of solutions of the minimal graph equation in the sequence of domains $\Omega_n$ with boundary data
$u_n= + \infty$ on $\gamma_n \cup \gamma_{-n}$ and $u_n =0$ on $\partial_{\infty} \Omega_n$, converge
uniformly to the horizontal section $\HY^2 \times \{ 0\}$.
\end{proposition}

\begin{figure}
\resizebox{7cm}{!}{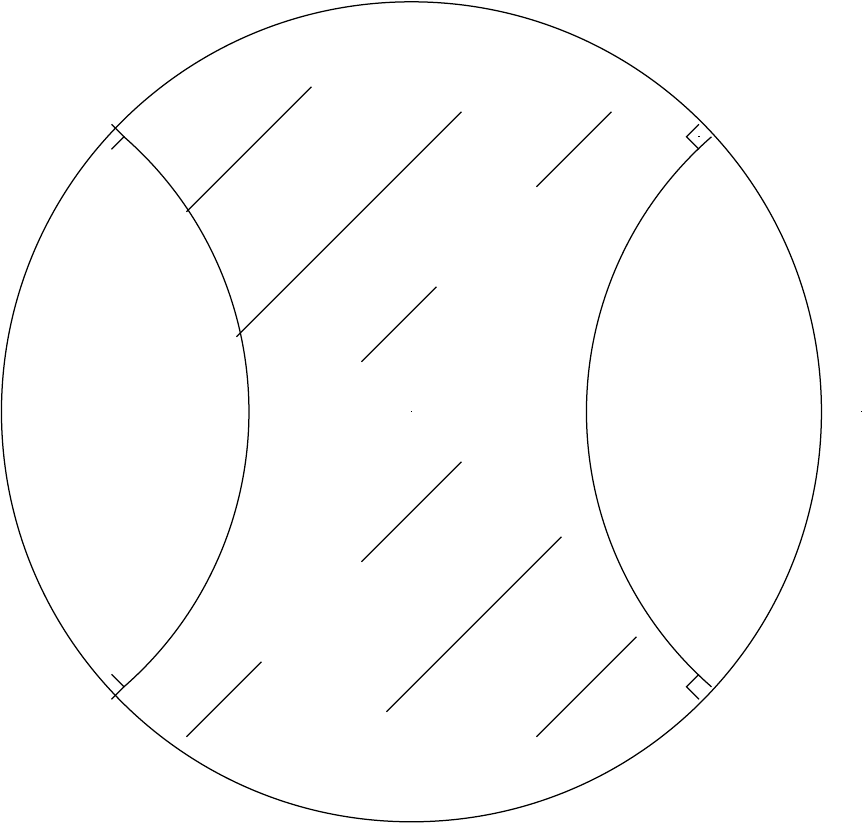 } \hfill
\resizebox{7cm}{!}{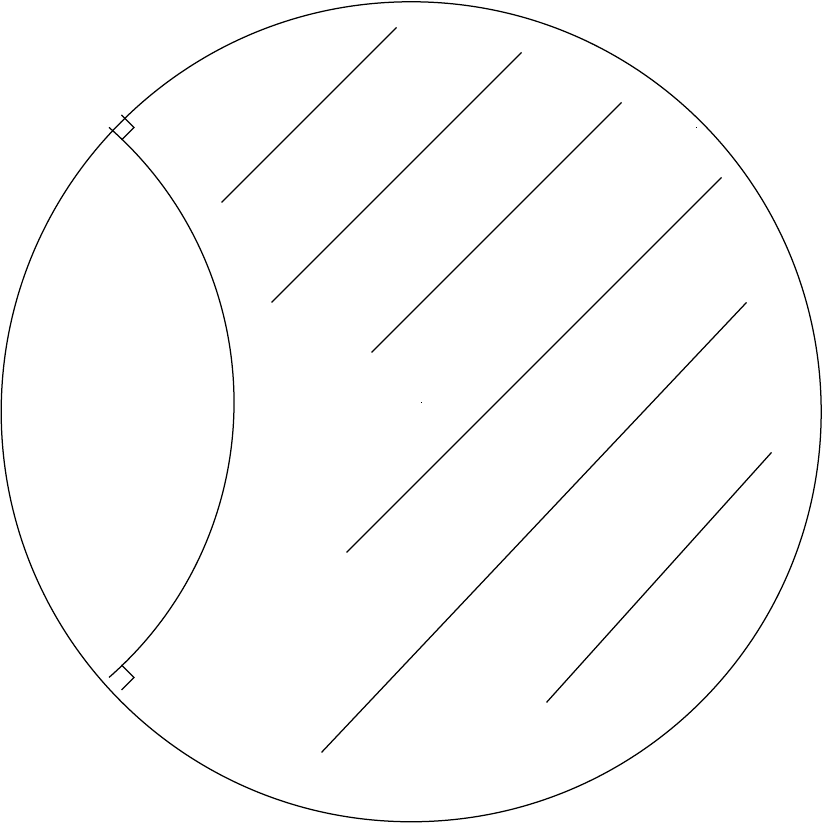}
\caption{Left: Domain $\Omega_n$ of function $u_n$. Right: Domain of functions $v_n$. }
\label{fig67}
\end{figure}

\begin{proof}
The sequence of domains $\Omega_n$ is an increasing sequence in $\HY^2$; $\Omega_{n} \subset \Omega_{n+1}$. The maximun
principle assures that the sequence  is decreasing with $0 \leq u_{n+1} (q) \leq u_n (q)$ for any $q \in \Omega_n$. Hence the
sequence of graphs $\Sigma_n$ converges to an entire graph of a function  $u_0: \HY^2 \to \R$. 
We will prove that $u_0 \equiv 0$. 

It suffices to prove that $u_0 = 0$ on the geodesic $\gamma_0$. If not we can assume
that $\sup_{\gamma_0} u_n = a_n \geq b >0$ and there is $p \in \gamma_0$ such that $u_0 (p)=b$. 

This point $p$
exists because $u_0$ takes value $0$ at infinity of $\gamma_0$. This comes from the fact that $(u_n)$ is a decreasing
sequence hence $u_n =0$ at infinity of $\Omega_n$ for any $n \in \N$.

We consider the sequence of minimal surfaces $\tilde \Sigma_n$ graphs  of a function $v_n$ on a domain $V_n$ bounded by $\gamma_n$,
with boundary data $v_n=+\infty$ on $\gamma_n$ and $v_n = b$ on $\partial_\infty V_n$. This family of graphs is well known and Mazet, Rodriguez and Rosenberg proved in \cite{M-R-R} that the sequence $v_n$ converges uniformly to $v_0 \equiv b$, when $n \to \infty$.

We restrict the function $v_n : V_n \to \R$ to the domain $W_n$ bounded by the geodesic $\gamma_n$ and the geodesic $\gamma_0$.
On $W_n$ we claim that the maximum principle applies to show that $v_n \geq u_0$. To see this its suffices to
check for the inequality on the boundary of the domain $W_n$. On $\gamma_n$, the function $v_n = + \infty > u_0$ and on
$\gamma_0$, we have $v_n \geq b \geq u_0$. At the boundary at infinity $\partial _{\infty} V_n$ we have $v_n= b > 0=u_0$.

Now let $n \to \infty$ to show that $v_n \to v_0=b$ and $v_0 \geq u_0$. This proves by symmetry that the function $u_0 \leq b$. 
The point $p$ of $\gamma_0$ where $u_0 (p)=b$ is an interior maximum point of the function, hence $u_0 =b$. This contradicts
the fact that $u_0$ take the value $0$ at the boundary at infinity of $\HY^2$.
\end{proof}

We consider an end $E_{(1,0)}(c) $ contained in a slab $S= \{ (x,y,t)\in \R^3 ; y \geq 1 \hbox{ and } -c_1 \leq t \leq c_1\}$ and we use the proposition \ref{graphesigma}  to obtain  ${\calC}^0$ convergence:

\begin{proposition}
\label{coordonee3}
An end $A$ of type $(p,0)$ of a properly immersed minimal surface $\Sigma$ in $M \times \SY^1$ has  third coordinate
which has a limit at infinity i.e. that $A$ converges in the ${\calC}^0$ norm to $A_{(p,0)}$ at height $a \in ]-c_1, c_1[$.
\end{proposition}

\begin{remark}
We will prove in the next proposition that $A$ converges uniformly in the  ${\calC}^2$ norm  to $A_{(p,0)}$ i.e; $A$ is a graph converging
uniformly to the cusp $A_{(p,0)}$ at height $\{ t=a\}$.
\end{remark}

\begin{proof}
A covering $E$ of $A$ is contained in a slab bounded by $E_{(1,0)} (-c_1)$ and $E_{(1,0)}(c_1)$. We study the intersection of $E$ with the level section
$E_{(1,0)}(c)=\{ (x,y,t)\in \R^3 ; y \geq 1 \hbox{ and }t=c\}$ with $c \in ]-c_1, + c_1[$. If $\Gamma$ is a compact component  of $A \cap E_{(1,0)}(c)$ then $\Gamma \cap \partial A \neq \emptyset$. Otherwise $\Gamma$ bounds a disc $D$ or a subend $A_0$. In  both cases, the maximum principle of proposition \ref{graphesigma} applies and $A=A_{(p,0)}$ i.e. $A$ is a flat standard end of height $c$.

Varying the value $c \in ]-c_1, + c_1[$, there is a value $a \in ]-c_1, + c_1[$, such that the intersection $A \cap E_{(1,0)}(c)$ has a non compact
component denoted by $\Gamma$. By proposition \ref{graphesigma}, $\Gamma \cap \partial E \neq \emptyset$. In the lift $E$ of $A$, we consider two lifts of
 $\Gamma$ denoted by $\Gamma_1$ and $\Gamma_2:= \psi \circ \Gamma_1$. The curves $\Gamma_1, \Gamma_2$
and a compact arc $\Gamma_3 \subset \partial E$  bound a fundamental domain, (see figure \ref{fig67}).

We consider a graph obtained by a translation $\sigma_n$ of $\Sigma_n$ of section \ref{graphebarriere} such that the geodesic $\gamma_{-n}$ translates to a fixed geodesic $\bar \gamma$ which does not intersect $\Gamma_3 \subset \partial E$ and $\Sigma_n$ is above $E$ with boundary data $u_n=a$ at the boundary at infinity. The graph $\Sigma_n$ has a line of curvature which is a graph over the translation of the geodesic $\gamma_0$ denoted by $\sigma_n \circ \gamma_0$ and the boundary curve
$\gamma_n$ is sent to $\sigma_n \circ \gamma_n$. We remark that $\sigma_n \circ \gamma_0$ is a distance $n$ from $\bar \gamma$ and  $\sigma_n \circ \gamma_n$ at a distance $2n$ from $\bar \gamma$.

We let $n \to \infty$ and fix the geodesic  $\bar \gamma$, using the horizontal
isometry $\sigma_n$. These graphs are conjugate to the graphs of the sequence $u_n$ of proposition \ref{graphesigma}. We know that the graph over $\sigma_n \circ\gamma_0$ is converging to the height $a$, hence we see that the end $A$ cannot have a point above the height $a$ at infinity. We do the same
with a symmetric graph with value $-\infty$ on the geodesic $\gamma_{-n}$ and $\gamma_{+n}$. This proves that the end $A$ is trapped between two graphs which have  third coordinate going to the same value $a$. Hence the end $A$ converges in the ${\calC}^2$ norm to a cusp  end $t=a$.
\end{proof}

\section{Proof of the theorem in $M \times \SY^1$.}

The surface $\Sigma$ is properly immersed. By lemma \ref{proper}, \ref{annulus} and theorem \ref{trapping}, each end $A$ lifts to $E$
a half-plane trapped between two standard ends $E_{(p,q)}$. First we prove the theorem for an end of type $(0,p)$ and then we adapt the arguments to the general case. 

{\bf Ends of type $(0,p)$; the vertical case.} Since $E$ is trapped between two vertical planes and the distance
between two vertical planes tends to zero as $y \to \infty$, we can assume $E \subset {\rm Slab} (\ell/2)$ where
$ {\rm Slab} (c) = \{ (x,y,t)\in \R^3; y>0 \hbox{ and } |t|\leq c \}$, and $\partial E \subset \{ y=y_0 \}$. We will prove
that when $p \in E$, and $y(p) > 3+ y_0$, then the killing field $\frac{\partial}{\partial x }$ is transverse
to $E$ at $p$. Then this sub-end of $E$ is stable hence has bounded curvature. We will prove later that this
gives the theorem in this case. 

Suppose on the contrary that some $p$, with $y(p) > 3+y_0$, has $\frac{\partial}{\partial x}|_p$ in $T_pE$.
The annulus $F_0$ meets $P(0)$ orthogonally and the normal vector to this curve of intersection
is in the plane $P(0)$ and takes all directions in this plane as one goes once around the curve.

Since $\frac{\partial}{\partial x} |_p \in T_pE$, the normal vector to $E$ at $p$ is in the plane $P(x=x(p))$.
Thus we can translate $F_0$ to $p$ (call $F_0$ this translated $F_0$) to be tangent to $E$ at $p$.
By a translation of less than $\ell/2$ we can assume $x(p)=0$, so now $E \subset {\rm Slab} (\ell)$.
Recall that $\gamma_1$ is the geodesic joining the centers of the boundary circles of $F_0$ and
$q_1=\gamma_1 \cap P(0)$. Write $q_1=(0,y_1,0)$ with $y_1>2+y_0$.

The convex hull of the foliation of ${\rm Tub}^- (F_0) \cup {\rm Tub}^+ (F_0) \cup F_0$ has $y$ coordinate at least the minimum of the $y$-coordinate of the boundary circles of $F(t)$ i.e. $y \geq y_0 +1/2$ on the convex hull.
Now we proved in the section \ref{annuli}, that this foliation contains a family of geodesic balls $B_\rho(q)$ of radius $\rho>0$ centered at points $q \in S_+ \cup S_-$. We choose this constant $\rho$ such that 
$$ 3 \rho < {\rm dist} (F_0, \gamma_1) \hbox{ and } 4 \rho <1.$$
Each such  geodesic ball $B(q)$ contains a compact annulus $\calC_\ell (q)$ bounded by geodesic circles of radius $\delta$ contained in $P(\ell)$ and $P(-\ell)$.

\vskip 0.5cm
{\it Step 1: Construction of arcs on $E$.} We know by proposition \ref{componentcatenoid}, that there are at
least two connected components $\Sigma_1, \Sigma_2$ of $E-F_0$ that have $p$ in their closure,
and $\Sigma_1, \Sigma_2 \subset F_0^-$. Clearly, by the maximum principle, each of $\Sigma_1,\Sigma_2$ intersects each of the catenoids in the local foliation $F_s$ about $F_0$ in $F_0^-$. In particular there is a
$\tilde q \in S_-$ such that ${\calC}_\ell(\tilde q) \cap \Sigma_1 \neq \emptyset$. Now translate ${\calC}_\ell(\tilde q)$
along the geodesic joining $\tilde q$ to $q_1$ and apply the Dragging lemma to obtain a point
$ p_1 \in \Sigma_1 \cap {\calC}_\ell (q_1) \subset B_{\rho} (q_1)$.

The same argument gives a point $p_2 \in \Sigma_2 \cap {\calC}_\ell (q_1)$. Recall that $p_1$ and $p_2$ can not be joined by an arc in $E \cap F_0^-$ (we will use this later).
Now we construct a loop $\mu$ in $E$.

For a value $k_0 \in \N$ which will be defined in step 2, we consider $\Gamma _+$ to be the euclidean segment joining $q_1= (0, y_1,0)$ to  $(0  , y_1, k_0 h + 2 \rho)$, together
with the segment joining $(0, y_1 , k_0 h + 2 \rho)$ to $z= (0, y_0, k_0 h +2 \rho)$.  
We will connect the point $p_1$ and $p_2$ by an arc in 
$E$ which stays in a tubular neighborhood of $\Gamma_+ \cup \partial E$. 
We note by ${\rm Tub}_\rho (\Gamma_+)$ the tubular neighborhood of geodesic radius $\rho$ along $\Gamma_+$.
We parametrize  the curve $\Gamma _+$ in a piecewise ${\calC}^1$-monotone manner by $q( \bar t), 0 \leq \bar t \leq 1$  and
we move $B_\rho (q_1)$ along  $q( \bar t)$, from $q_1$ to $z= (0, y_0, k_0 h +2 \rho)$, by $B_\rho (q(\bar t))$.
Each ball $B_\rho (q(\bar t))$, $q \in \Gamma _+$ contains the catenoid ${\calC}_\ell (q(\bar t))$
 and the Dragging lemma then gives two continuous paths $\sigma_1^+ (\bar t), \sigma _2^+ (\bar t)$
starting at $p_1$,$p_2$ respectively such that $\sigma_i^+ (\bar t) \in E$ for $0 \leq \bar t \leq 1$.

We apply the Dragging lemma up to the value $q(1)=z$ and $ \sigma_i^+(1) \in B_\rho(z)$ for $i=1,2$. Since  $\tilde p_1= \sigma_1^+ (1)$ and $\tilde p_2= \sigma_2^+ (1)$ are in $\partial E \cap B_\rho (z)$, we can find a path $\sigma_{12}^+$ in $\partial E$ from $\tilde p_1$ to $\tilde p_2$. We have
$t(\tilde p_1), t(\tilde p_2) \in [k_0 h + \rho, k_0 h+ 3\rho]$. We will prove in step 2, that we can find
a path $\sigma_{12}^+ \in \partial E$  from $\tilde p_1$ to $\tilde p_2$ such that for all $p \in \sigma_{12}^+$, $t(p) \in [\rho, k_0 h + 3\rho]$.

Assuming this, we have constructed a path $\mu ^+$ in $E$ from $p_1$ to $p_2$ which is
$$\mu ^+ = \sigma_1^+ \cup \sigma_{12}^+\cup \sigma_2^+.$$

The arcs  $\sigma_1^+ (\bar t), \sigma_2^+ (\bar t)$ are contained in $T_\rho (\Gamma_+)$. The arcs of $\sigma_1 ^+ $ and $\sigma_2 ^+ $ from $p_1$ to $F_0$ and $p_2$ to $F_0$ are disjoint (see proposition \ref{componentcatenoid}) since 
$\sigma_1^+ \subset \Sigma_1$ and $\sigma_2^+ \subset \Sigma_2$ in ${\rm Tub}^- (F_0)$.

Moreover the paths are quasi-monotone along the segment of $\Gamma _+$ in ${\rm Tub} (\Gamma _+)$: once
the catenoids ${\rm C}_\ell (q) , q \in \Gamma _+$ have advanced along $\Gamma _+$ a distance $2\rho$, the paths $\sigma_1^+$ and $\sigma_2^+$ do not return to the $\rho$-ball where they started.

If the arcs $\sigma_1^+ (\bar t)$ and $\sigma_2^+ (\bar t)$ remain disjoint for $\bar t \leq 1$, we do not change $\mu^+$.
If the arcs intersect then at the first point of intersection $p_3$ we replace $\mu^+$ by the path on $\sigma^+_1$ from $p_1$ to $p_3$ union the path on  $\sigma^+_2$ from $p_2$
to $p_3$. Such a point $ p_3$ is necessarily outside $F_{-1/2}$.

\begin{figure}
\resizebox{16cm}{!}{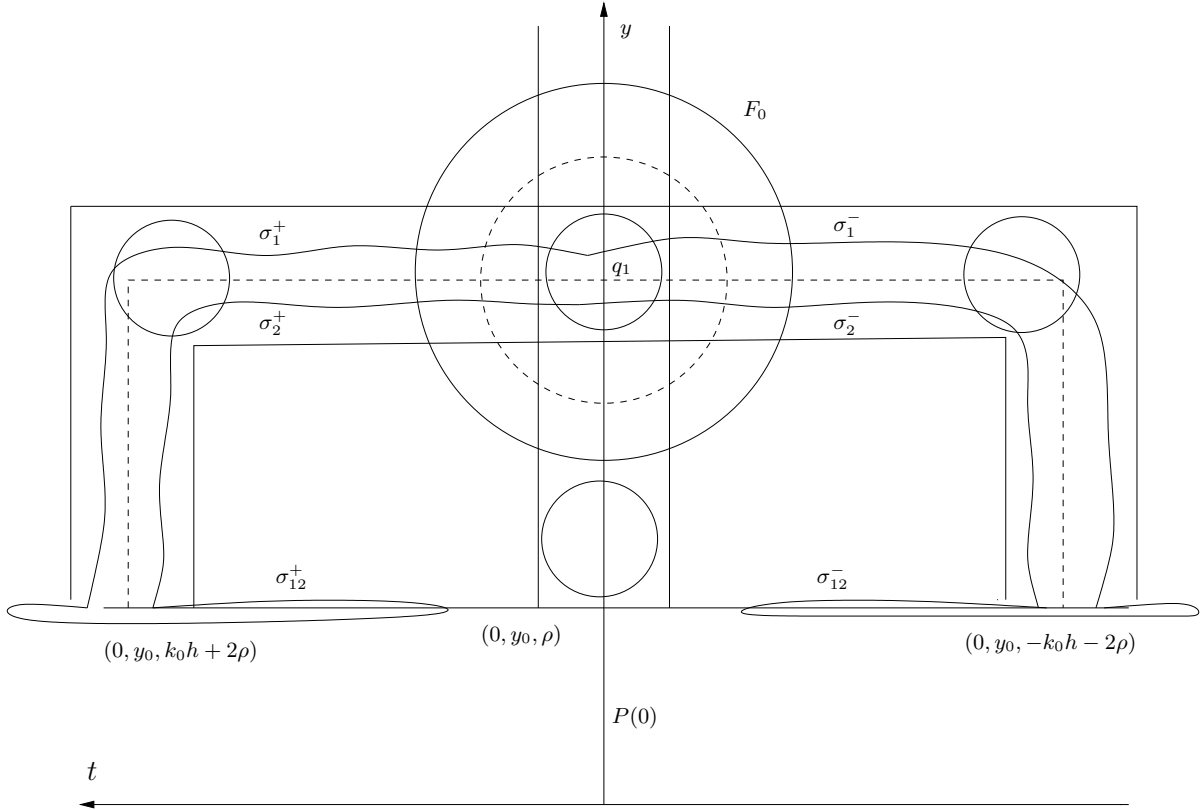 }
\caption{Construction of the arc $\Gamma$ }
\label{fig16}
\end{figure}

{\it Step 2: The boundary $\partial E$.} Now we study the boundary of the annulus and the function $t : \partial E \to \R$ the restriction of the third coordinate in the model of the half-plane.  We parametrize the boundary curve $\partial E$ by the immersion $C: \R \to \HY^2 \times \R, C(s)=(x(s),y_0,t(s))$
with period 
$$C(s+1)= (x(s+1), y_0, t(s+1)) \longrightarrow (x(s), y_0, t(s)+h).$$
The diameter is defined by
$$G:=\sup_{s_1,s_2 \in [0,1]} |t(s_1)-t(s_2)|$$
and choose $k_0\in \N$ such that $K=k_0 h \geq G$. We consider the intersection of $\partial E$ with a transverse plane to the curve
$P(\alpha):=\{ (x,y,t)\in \R^3; y=y_0, t=\alpha \}$. Since $C$ is a proper immersed curve, we have a finite number of intersection 
points
$$C(s) \cap P (\alpha)= \{ C(s_1),...,C(s_\ell) \}.$$
We claim that $(s_i)_{1\leq i \leq \ell} \in [ s_1-k, s_1 + k]$. To see this we remark that if $s_1 +1+k' \geq s \geq s_1+k' \geq s_1 +k$, we have

$$t(s) - \alpha = t(s) -t(s_1 + k') +  t(s_1 + k') - t(s_1) \geq k'\tau -G \geq k_0\tau -G >0.$$
Hence independently of the choice of $\alpha$, two points of $\partial E$ with the same $t$ coordinate are connected by 
a sub-arc $\Gamma$ of $\partial E$ with $t(\Gamma) \subset [\alpha-K, \alpha + K]$. Two points of $\partial E$
with coordinate $t_1 \leq t_2$ can be connected in $\partial E$ by a sub-arc $\Gamma$ with $x(\Gamma) \in [t_1 -K,t_2+K]$.

\vskip 0.5cm
{\it Step 3: A loop $\mu$ in $E$.} In step 1, we constructed an arc $\mu^+= \sigma_1 ^+ \cup \sigma_{12}^+\cup \sigma_2 ^+ $ which joins the points $p_1$ and $p_2$ and $\mu ^+  \subset {\rm Tub}_\rho (\Gamma _+) \cup \partial E$. Now do this
construction in the half-space $\{x \leq 0\}$ to obtain a path $\mu ^-$ joining $p_1$ to $p_2$ with similar
properties. Let $\Gamma _-$ be the segment from $q_1= (0, y_1,0)$ to  $(0  , y_1,- k_0 h - 2 \rho)$, together
with the segment joining $(0, y_1 ,-  k_0 h - 2 \rho)$ to $z= (0, y_0,-  k_0 h - 2 \rho)$.  Move $B_\rho (q_1)$
in a ${\calC}^1$-monotone manner along $\Gamma _-$ and we use the Dragging lemma as before to construct arcs $\sigma_1 ^- (\bar t), \sigma_2 ^- (\bar t)$ in $E$. We note by ${\rm Tub}_\rho (\Gamma _-)$ the tubular neighborhood of geodesic radius $\rho$ along $\Gamma _-$. We follow the arc 
$\sigma_1 ^- (\bar t), \sigma_2 ^- (\bar t)$ up to points of $\partial E$.  As in step 2, we construct the arc $ \sigma_{12}^-$, so that points
$p\in \sigma_{12}^-$ have coordinate $t(p) \subset [-k_0 h - \rho, -\rho]$.

 Finally we consider; see figure \ref{fig16}
$$\mu ^-= \sigma_1^- \cup \sigma_{12}^- \cup \sigma_2 ^-.$$
and we let $\mu$ be the loop $\mu ^+ \cup \mu ^-$. $\mu$  is contained in ${\rm Tub} _\rho (\Gamma _+ \cup \Gamma _-)$.
If the arcs $\sigma_1^- (\bar t)$ and $\sigma_2^- (\bar t)$ remain disjoint for $\bar t \leq 1$, we do not change $\mu^-$.
If the arcs intersect then at the first point of intersection $p_4$ we replace $\mu^-$ by the path on $\sigma^-_1$ from $p_1$ to $p_4$ union the path on  $\sigma^-_2$ from $p_2$
to $p_4$. Such a point $ p_4$ is necessarily outside $F_{-1/2}$.

The end $E$ is an immersed half-plane $X: \Omega=\{(u,v) \in \R^2; v \geq 0\} \to \HY^2 \times \R$ 
with $X(\Omega) =E$. The loop $\mu \subset E$ is immersed and we denote by $\hat \mu= X^{-1} (\mu)$ the pre-image of $\mu$ in $\Omega$. 

In the Dragging lemma, we  constructed  the arc  $\mu$ locally and then extended it. The 
preimage $\hat \mu$ is locally embedded in $\Omega$. The arc $\hat \mu$ can have   self-intersections.  If $\hat p$ is one of them, we consider the sub-arc $\gamma$ of $\hat \mu$ with end points $\hat p$. This sub-arc $\gamma$ bounds a disk in $\Omega$.
We  remove these sub arcs to  obtain a piecewise $\calC^1$ connected curve in $\Omega$ without self-intersecting points. This defines a closed Jordan curve which bounds a disk $D$ in $\Omega$. The immersion $X(D)$ is a minimal disk in $\HY^2 \times \R$ with boundary an immersed connected curve contained in ${\rm Tub}_\rho (\Gamma _+ \cup \Gamma _-) \cup \partial E$. Now we analyze the geometry of the disk $X(D)$.

Consider the plane defined by $P(0)= \{ (x,y,t)\in \R^3; y \leq y_0 \hbox{ and } t=0 \}$. This plane
separates  $\Gamma _+ \cup \Gamma _-$ in two connected components.

We denote by $\tilde \mu =\partial X(D)$ the boundary of the minimal disk. Let  
 $\tilde \mu _1=\tilde \mu  \cap (\sigma _1^+ \cup \sigma _1 ^-)$  and  $\tilde \mu _2=\tilde \mu  \cap (\sigma _2^+ \cup \sigma _2 ^-)$ be  the connected components  of the loop in $E$
 containing $p_1$ and $p_2$ respectively. The end points of $\tilde \mu _1, \tilde \mu _2$
 are in different half-spaces determined by $\{x=0\}$ (one end point has $x>\rho$ and the other
 $x <-\rho$). Thus the plane $P(0)$ intersects  $\tilde \mu _1$ and  $\tilde \mu _2$, each one in an odd number of points.

Now we will obtain a contradiction by proving that $P(0) \cap \tilde \mu _1$ is an even number of points. 
One translates  horizontal catenoids $\calC_\ell (q) $, $q \in E(0) \cap P(0)$, starting far from $\mu$ to see that before $\calC_\ell $ touches a $\rho$-tubular neighborhood of $\mu$, one does not touch the disk $X(D)$. Hence  $X(D) \cap P(0)$ is contained in $F_0^- \cap {\rm Tub}_\rho (\Gamma_+ \cup \Gamma_-)$.

In $F_0^-$ the sub arc  $\tilde \mu _1 \subset \Sigma_1$ and  $\tilde \mu _2 \subset \Sigma_2$ cannot be connected. 
Hence a connected arc  $\gamma \subset X(D) \cap P(0)$
must have end points either in $\Sigma_1$ or in $\Sigma_2$. This means that there are an  even number of point of $\tilde \mu _1 \cap P(0)$   on $\partial D= \mu$. This contradicts the odd intersection number of each arc with $P(0)$.

This proves that $E$ is a graph for $y \geq y_0 + 2R$.

{\bf Ends of type $(p,q)$, tilted planes.} Next we prove the theorem when $E$ is trapped between
two tilted (not horizontal) planes $E(p,q)$. We can suppose $E$ is contained a tilted slab $S$ of the form, for some $c_1 >0$:
$$S= \{ (x,y,t) \in \R^3; y >0 \hbox{ and } -c_1 \leq p \tau t - qh x \leq c_1 \}.$$
Since $S$ is converging to a vertical slab as $y \to \infty$, there is a $y_0 >1$ so that if $p \in E$,
$y(p) \geq y_0$, then the catenoid $C_\ell (p)$ in $B_\rho (p)$, has both its boundary circles outside
of $S$. To see this, we use an isometry which leaves the slab $S$ invariant and takes $p \in E$ to a point 
$\tilde p=(x,y,0)$. Observe that $S \cap \{ |t| \leq 1 \}$ is in a vertical slab bounded by
$P(-c_2)$ and $P(c_2)$, where $c_2$ depends on $p\tau$ and $q h$. Then for any point $\tilde p=(x,y,0)$ with
$|x| \leq c_2$,  $\calC_\ell (\tilde p) \subset B_\rho (\tilde p)$ has boundary circles outside of the slab bounded by $P(d)$ and $P(-d)$ for $d \geq c_2$, and $y$ greater than some $y_0$ (using that $(x,y,t) \to (\lambda x, \lambda y,t)$ is an isometry). This property is invariant by changing $\tilde p=(x,y,0)$ to $p=(x+p\tau, y ,t +qh)$. 

We will prove that a sub-end of $E$ is transverse to $\frac{\partial}{\partial x}$ for large $y$. Suppose
this is not the case. We proceed exactly as in the case $E$ is trapped between vertical planes
to find $p_1, p_2 \in E \cap B_\rho (q_1)$, $q_1$ the center of a horizontal catenoid $F_0$, and $p_1,p_2$ can not be joined by a path in $E$ that is inside $F_0$. The proof is modified in our choice of 
$\Gamma=\Gamma_+ \cup \Gamma_-$, and a loop in $E$ passing through $p_1$ and $p_2$.

We denote by $\vec u $ the unit vector director
of the straight line $\{ (x,y,t) \in \R^3;  p\tau t - qh x=0 \hbox{ and } y=y_0 \}$.
For a value $k_0 \in \N$ which depends on the diameter of the periodic boundary curve, 
we consider $\Gamma _+$ be the euclidean segment joining $q_1= (0, y_1,0)$ to  $q_1 + (k_0 h + 2 \rho) \vec u$, together
with the segment joining $q_1 + (k_0 h + 2 \rho) \vec u$ to $z= (0, y_0,0) +( k_0 h +2 \rho)\vec u$.
We connect the point $p_1$ and $p_2$ by an arc in  $E$ which stays in a tubular neighborhood of 
$\Gamma_+ \cup \partial E$.

Let $\Gamma _-$ be the segment from $q_1= (0, y_1,0)$ to  $(0  , y_1,0) - (k_0 h + 2 \rho) \vec u$, together
with the segment joining $(0, y_1 ,0) -  (k_0 h  + 2 \rho)\vec u$ to $z= (0, y_0,0)- ( k_0 h + 2 \rho) \vec u$. 
We connect the point $p_1$ and $p_2$ by an arc in  $E$ which stays in a tubular neighborhood of 
$\Gamma_- \cup \partial E$.

Then the argument is the same to obtain a contradiction with $\Gamma= \Gamma_+ \cup \Gamma_-$. 

{\bf Ends of type $(p,0)$, horizontal planes.}  Let $E$ be a half-plane end (lifting of $A \subset \calM$)
to $\HY \times \R$, between the planes $t=\pm d$, with $\partial E \subset \{ y=1\}$, $\partial E$ invariant
by the isometry $(x,y,t) \to (x+ \tau, y,t)$. By proposition \ref{coordonee3}, we can assume $t \to 0$ on $E$
as $y \to \infty$. So for $y_0$ large, the sub-end of $E$ given by $y \geq y_0$ is between planes $t= \pm c$
for any small $c>0$.

Let $\eta$ be a circle of radius one in $\{t=c\}$ and let $\eta_-$ be $\eta$ translated vertically
to a circle in $\{ t =-c\}$. for $c$ small enough, $\eta \cup \eta_-$ bounds a stable (rotational) annulus $F_0$.
$F_0$ is a bigraph over $\{ t=0\}$. Now we assume $y_0$ chosen so that $E$ is between $t= \pm c$ for $y \geq y_0$
and then $\partial F_0 \subset \{ t = \pm c \}$.

As in section 6, where $E$ was trapped between two vertical planes and $F_0$ was a horizontal catenoid, we define
$B_\rho (q)$, $\calC_\ell$ in the same manner, with $\calC_\ell$ a vertical catenoid. We choose $y_0$ large enough
so that $E$ is between $t= \pm \ell/2$ and $\calC_\ell$ has its boundary circles in $t= \pm \ell$ for $y \geq y_0 +3$.

Suppose $p$ is in $E$ , $y(p) \geq y(0) + 3$, and $E$ has a vertical tangent plane at $p$.
Then one places a vertical catenoid $F_0$ to be tangent to $E$ at $p$
(after a small translation) and one obtains $p_1,p_2 \in E \cap B _ \rho (q)$, $q$ the center of $F_0$, such that
$p_1,p_2$ can not be joined by a path in $E$ that is inside $F_0$.

For a value $k_0 \in \N$ which depends on the diameter of the periodic boundary curve, 
we consider $\Gamma _+$ be the euclidean segment joining $q_1= (0, y_1,0)$ to  $(k_0 h +2 \rho,y_1,0)$, together
with the segment joining $(k_0 h +2 \rho,y_1,0)$ to $z= (k_0 h +2 \rho, y_0,0)$.
We connect the points $p_1$ and $p_2$ by an arc in  $E$ which stays in a tubular neighborhood of 
$\Gamma_+ \cup \partial E$.

Let $\Gamma _-$ be the segment from $q_1= (0, y_1,0)$ to  $( -k_0 h - 2 \rho , y_1,0) $, together
with the segment joining $( -k_0 h - 2 \rho , y_1,0) $ to $z= ( -k_0 h - 2 \rho , y_0,0) $. 
We connect the point $p_1$ and $p_2$ by an arc in  $E$ which stays in a tubular neighborhood of 
$\Gamma_- \cup \partial E$. We apply now the same argument to obtain a contradiction.

{\bf  Finite total curvature.} 
We proved that a minimal annulus is trapped in Slab and is
a killing multigraph outside a compact set $K_0 \subset M \times \SY^1$. 
These graphs are stable, hence they have bounded Gaussian curvature. 
They are contained in a euclidean slab whose hyperbolic width tends to zero at infinity. 

In the horizontal case with $A$ asymptotic to $A_{(p,0)}$, the end $A$ has a limit for its third coordinate. 
Since the curvature is bounded, $A$ is a vertical graph of a function $f : A_{(p,0)} \to \R$, with $f$ converging to $0$ in a 
$\calC^2$ manner. The end $A$ is converging to the cusp $\calC \times \{0\}$ and the curve $\bbT (y) \cap A=\gamma (y)$ is a topological circle converging to a finite covering of a quotient $c(y)/ [\psi]$. The curve $\gamma (y)$ has uniform bounded curvature and its length goes to zero. Thus  $\int _{\gamma (y)} k_g ds \to 0$ as $y \to \infty$.

In the case of ends of type $(0,q)$ and $(p,q)$, the ends are horizontal multi-graphs on some
$A(0,p)$. Since $A$ converges in a $\calC ^2$ manner to $A(p,0)$, the  curves $\gamma (y)= \bbT (y) \cap A$ converge 
to a finite covering of a quotient of a vertical geodesic by the translation $T(h)$. This implies that the curvature of 
$\gamma (y)$ converges uniformly to zero as $y \to \infty$.

We apply the Gauss-Bonnet formula on an exhaustion of $M \times \S^1$ by a sequence of compact $K_n$, with boundary
of $K_n$  the union of mean curvature one tori $\bbT_1(n), ...,\bbT_k(n)$, in each end ${\calM} \subset M \times \SY^1$ and $\gamma_{k,n} = \bbT_k(n) \cap \Sigma$.

$$\int_{K_n \cap S} K dA + \int _{\gamma (k,n)} k_g ds = 2 \pi \chi (\Sigma).$$
When $n \to \infty$, the integral of the curvature on $\gamma (k,n)$ tends to zero and we obtain the finite total curvature
formula
$$\int_{ S} K dA = 2 \pi \chi (\Sigma).$$

\section{Proof of the theorem in $N$}

Now we complete the proof of the Theorem 1.1 when the ambient space is $N$. The idea is the same as in $M \times \SY^1$. Let $A$ be an annular end in ${\calM}(-1)$, minimal and properly immersed. By Lemma 4.1
(the same proof) we can suppose $A \subset \cup_{y \geq 1} \bbT (y)$, $\partial A$ is an immersed closed curve
and $A$ is transverse to $\bbT (1)$ along $\partial A$. Let  $E$ be a connected lift of $A$ to $\HY^3$,
so $\partial E \subset \{ (x,1,t) \in \R^3\}$. Observe that $E$ is a half-plane, not an annulus. Suppose, on the
contrary that $E$ is an immersed annulus, so $\partial E \subset \{ (x,1,t) \in \R^3\}$ is compact. Let $D$
be the convex hull of $\{ (x,0,t) \in \R^3; (x,1,t)\in \partial E \}$ in the $y=0$ plane. Let $L$ be a line of the plane
$Q=\{ y=0\}$, disjoint from $D$.

Let $C$ be a small circle in $Q$ in the half-space $\calH$ of $Q-L$ disjoint from $D$. $C$ bounds a totally
geodesic hyperbolic plane in $\HY^3$ (it is a hemisphere orthogonal to $Q$ along $C$ in our model). For $C$
small, this plane is disjoint from $E$. Let the circle $C$ grow in $\calH$ and converge to $L$. By the maximum
principle, there is no first contact point of the planes bounded by these circles with $E$ (the planes do not touch
$\partial E$). Since the hyperbolic planes bounded by the circles converge to $L \times \R^+$, it follows
that $E$ is on one side of $L \times \R^+$. Hence $E$ is contained in the cylinder ${\rm Cyl}= \{ (x,y,t) \in \R^3 ; (x,0,t) \in 
\partial D, y >0 \}$.

For $y$ large, the diameter of ${\rm Cyl}$ tends to zero; i.e. the diameter of ${\rm Cyl} \cap \{ y ={\rm const} \}$ tends to zero. So we could
touch $E$ by a catenoid at an interior point of $E$; a contradiction.

Now we know $E$ is a half-plane. After an isometry of $\HY^3$, we can assume $\partial E$ is invariant
under the parabolic isometry: $(x,y,t) \to (x + \tau, y,t)$, and $\partial E \subset \{ y=1\}$. so the $t$ coordinate is bounded on 
$\partial E$. The same convex hull argument as in the previous annular case, then shows the $t$
coordinate has the same bound on $E$; $|t| \leq c$, for some $c>0$, (one takes $L$ to be a horizontal line in $\{ y=0\}$,
above height $c$, and considers circles $C$ in $\{y=0\}$ above $L$. When $C$ converges to $L$
in $\{ y=0 \}$, the hyperbolic planes in $\HY^3$ bounded by $C$, are disjoint from $E$ and converge to
$L \times \R^+$). So $E$ is trapped between two horizontal planes $t = \pm c$.

The distance between these horizontal planes tends to zero as $y \to \infty$. Now we will prove that for
$y$ large, the killing field $\frac{\partial}{\partial t}$ is transverse to $E$. hence a sub-end of $E$ has bounded
curvature.  This will complete the proof as follows.  The sub-end is a vertical graph over the plane $t=0$, that converges to zero in the ${\calC}^2$-topology.
The graph function is the distance to the plane $t=0$. Thus the geodesic curvature of the curve in $E$, given
by ${\rm Cyl}=E \cap \{ y={\rm constant}\}$ is bounded. Also the length of this curve $C_y$ tends
to zero in $C_y$ modulo $(x,y,t) \to (x+ \tau, y ,t).$ This yields the formula for the finite total curvature of
$\Sigma$ in $N$: apply Gauss-Bonnet to the compact part of $\Sigma$ bounded by the curves $C_y$ in the ends
and let $y \to \infty$.

Thus it suffices to prove $E$ is transverse to $\frac{\partial }{\partial t}$ for $y$ large. The proof of this is the same as in section 8, for an end trapped between two horizontall planes.  More precisely, for an end $E$ in $\calM$ between two horizontal planes that are close,  the distance between
the planes $|t| =c$ tends to zero as $y \to \infty$, so one can put a vertical catenoid $F_0$, whose boundary circles are of radius one and in the horizontal planes $|t|=d >c$, when the center $q$ of $F_0$ has $y(q)$ larger than some
$y_0$.

One chooses $\rho,\ell$ as in Sections 6 and 8, and using the Dragging lemma, one shows that if $E$
has a vertical tangent plane at $p$, $y(p)$ large, then one finds $p_1,p_2 \in B_\rho (q) \cap E$, that can not be
joined by a path in $E \cap F_0^-$. One defines $\Gamma= \Gamma_+ \cup \Gamma_-$ and the same proof now gives a contradiction.

\end{document}